\documentclass[AMA,STIX1COL]{WileyNJDv5} 
\usepackage{amssymb}
\usepackage{amsmath}

\usepackage{mathrsfs}
\usepackage{mathptmx} 
\usepackage{newtxtext,newtxmath} 

\usepackage{supertabular}
\usepackage{lineno,hyperref}
\usepackage{siunitx}
\usepackage{soul}
\usepackage{tikz}
\usepackage{comment}
\usepackage{graphicx,subcaption}
\usepackage[section]{placeins}
\usepackage{geometry}
\usepackage{subcaption}
\usepackage{bm,amsmath,amssymb}
\usepackage{tikz}
    \usetikzlibrary{arrows, positioning,shapes}
\usepackage{mathrsfs}
\usepackage{amsthm}

\usepackage{wrapfig}
\usepackage{lscape}
\usepackage{rotating}
\usepackage{epstopdf}

\newcommand{\ie}{\textit{i.e.},\ }
\newcommand{\eg}{\textit{e.g.},\ }
\newcommand{\kins}{\kappa_\text{ins}}
\newcommand{\kcond}{\kappa_\text{cond}}

\newcommand{\dd}[2]{\frac{\partial #1}{\partial #2}}

{}

\articletype{Article Type}%

\received{Date Month Year}
\revised{Date Month Year}
\accepted{Date Month Year}
\journal{International Journal for Numerical Methods in Engineering}
\volume{00}
\copyyear{2023}
\startpage{1}

\begin{document}

\title{Improving the Robustness of the Projected Gradient Descent Method for Nonlinear Constrained Optimization Problems in Topology Optimization}

\author[mymainaddress]{Lucka Barbeau}
\author[mymainaddress]{Marc-Étienne Lamarche-Gagnon}
\author[mymainaddress]{Florin Ilinca}

\authormark{Barbeau \textsc{et al.}}
\titlemark{Improving the Robustness of the Projected Gradient Descent Method for Nonlinear Constrained Optimization Problems in Topology Optimization}

\address[mymainaddress]{\orgdiv{Simulation and Numerical Modeling}, \orgname{Automotive and Surface Transportation Research Center, National Research Council}, \orgaddress{\state{Quebec}, \country{Canada}}}

\corres{Corresponding author Lucka Barbeau, 75 Bd de Mortagne, 
            Boucherville,
            QC, 
            J4B 6Y4,
            Canada \email{lucka.barbeau@nrc-cnrc.gc.ca}}

\presentaddress{This is sample for present address text this is sample for present address text.}


\abstract[Abstract]{The Projected Gradient Descent (PGD) algorithm is a widely used and efficient first-order method for solving constrained optimization problems due to its simplicity and scalability in large design spaces. Building on recent advancements in the GD algorithm—where an inertial step component has been introduced to improve efficiency in solving constrained optimization problems—this study introduces two key enhancements to further improve the algorithm’s performance and adaptability in large-scale design spaces. First, univariate constraints (such as design variable bounds constraints) are directly incorporated into the projection step via the Schur complement and an improved active set algorithm with bulk constraints manipulation, avoiding issues with min--max clipping. Second, the update step is decomposed relative to the constraint vector space, enabling a post-projection adjustment based on the state of the constraints and an approximation of the Lagrangian, significantly improving the algorithm's robustness for problems with nonlinear constraints. Applied to a topology optimization problem for heat sink design, the proposed PGD algorithm demonstrates performance comparable to or exceeding that of the Method of Moving Asymptotes (MMA), with minimal parameter tuning. These results position the enhanced PGD as a robust tool for complex optimization problems with large variable space, such as topology optimization problems.}


\jnlcitation{\cname{%
\author{Barbeau L.}, 
\author{Lamarche-Gagnon M.-É.}, and
\author{Ilinca F.}}.
\ctitle{Improving the Robustness of the Projected Gradient Descent Method for Nonlinear Constrained Optimization Problems in Topology Optimization.} 
\cjournal{\it International Journal for Numerical Methods in Engineering.} \cvol{2024;00(00):1--18}.}

\maketitle

\renewcommand\thefootnote{}

\renewcommand\thefootnote{\fnsymbol{footnote}}
\setcounter{footnote}{1}

\bmsection{Introduction}
\label{sec::intro}
    

Topology optimization is a mathematical tool that determines the optimal structure or material distribution for a given task. It has seen a significant increase in interest due to advancements in manufacturing technologies, such as 3D printing, which can produce the complex geometries that often result from this process \cite{feng2018review}. Applications of topology optimization algorithms span a diverse range of fields, including structural engineering \cite{eschenauer2001topology}, heat transfer \cite{dbouk2017review}, and fluid dynamics \cite{alexandersen2020review}, among others. At the core of the topology optimization algorithm is the optimizer that minimizes a given cost function by iteratively changing the topology while preserving certain constraints related to the geometry (volume, perimeter projection, etc.) or some physical properties of the topology (compliance, pressure drop, etc.). A large variety of algorithms exist to solve such constrained optimization problems. The most known and widely spread in the field of topology optimization are the Method of Moving Asymptotes (MMA) \cite{svanberg1987method}, the Sequential Linear Programming (SLP) \cite{gomes2011slp}, the Sequential Quadratic Programming (SQP) \cite{rojas2016efficient}, the optimality criteria \cite{hassani1998review}, the Interior Point Method (IPM) \cite{hoppe2002primal}, and the Projected Gradient Descent (PGD) \cite{tavakoli2014multimaterial, nishioka2023inertial}.
Among these, projected gradient descent has seen a recent increase in interest, along with other first-order methods, due to its simplicity, reduced cost, and better scaling compared to second-order methods \cite{ nishioka2023inertial}. The PGD method benefits from having a generally small set of tunable parameters, making it easy to employ. For example, the work of Chzhen \textit{et al.} modified and improved this algorithm to make it parameter-free \cite{chzhen2023parameter}. Furthermore, work was done to improve the efficiency of this algorithm for topology optimization problems by using an acceleration method based on the inertial principal \cite{nishioka2023inertial,he2022boosting}. However, considering multiple constraints directly in the projection step leads to a general quadratic programming (QP) problem requiring appropriate methods to solve (\ie augmented Lagrange multiplier, active set method, interior point method). 
However, the projection step in the PGD algorithm involves solving a general convex QP problem, often a computationally intensive task. To manage this complexity, an active set approach is employed, which reduces the size of the projection matrix by focusing only on the active constraints \cite{gu2024random}. The active set methodology was effectively applied in a PGD context by Dostal \textit{et al.} \cite{dostal2005minimizing}, and since then, substantial advances have refined how active sets are incorporated within projections. One of the key developments is the infeasible strategy, where the active set of constraints is dynamically adjusted by adding constraints that are violated and removing those with negative Lagrange multipliers. This technique allows the algorithm to proceed from an initially infeasible point, gradually refining the feasible region until it reaches optimality \cite{kunisch2003infeasible,hungerlander2015feasible,forsgren2015active}. To accelerate the convergence, some algorithms introduce methods to deal with multiple constraints simultaneously. Schittkowski \cite{schittkowski2009active} proposes an approach that combines bulk constraint manipulation with a step length adjustment, removing redundant constraints based on activity duration to avoid cycling. Cristofari \textit{et al.} \cite{cristofari2017two}  proposed a two-stage algorithm for univariate constraints (bounds constraints) that enables backtracking and readjustment of the active set if updates fail to improve the objective, followed by a line search to ensure progress toward the solution. Additionally, to accelerate the definition of the active set, stochastic algorithms have also seen some development in this approach, manipulating multiple constraints simultaneously while ensuring, due to the stochastic nature of the algorithm, the convergence of the method \cite{gu2024random}. 

However, the robustness of the PGD method, especially in the presence of nonlinear constraints, can still be improved, which would improve the applicability of this algorithm to complex topology optimization problems. The use of nonlinear constraints in topology optimization problems is becoming increasingly frequent. As an example, the addition of overhang constraints \cite{gaynor2014topology,qian2017undercut,lamarche2024additively}, which aims at enhancing the manufacturability of the design of the part through topology optimization, introduces nonlinear constraints that can affect the convergence of the optimizer \cite{gaynor2016topology}. Some algorithm modifications are required to handle these constraints in the PGD method. The active set approach generally also requires either a stochastic algorithm to manipulate multiple constraints simultaneously or a sequence sequential manipulation of constraints, which can drastically increase the cost.
Additionally, there is room for improvement in handling univariate constraints to enhance the efficiency of the PGD algorithm. Univariate constraints (usually bounds constraints), which apply to only one of the design variables, are ubiquitous in topology optimization problems due to the requirement of maintaining the phase indicator within the bound of the material model used (usually between 0 and 1 for density-based approaches). As the discretization of the physics of the topology optimization problem at hand leads to a large number of these constraints, they are often enforced using a clipping method \cite{beck2017first}, which is not as efficient as considering them directly within the projection step \cite{zhang2020improved}. 

This work presents two improvements to the classical formulation of the PGD algorithm. First, an active set approach that enables bulk manipulation of constraints for a point projection problem is introduced. The algorithm uses an infeasible approach with bulk constraint manipulation combined with a backtracking approach when convergence cannot be guaranteed for one iteration. The projection step on the active set then uses a Schur complement that manages univariate and global constraints separately, allowing the algorithm to solve the interaction between global and univariate constraints effectively. While using the Schur complement to simplify large constraint systems is well-established \cite{bartlett2006qpschur,wong2011active}, particularly in QP where block matrices arise, its application to simplifying univariate constraints remains, to our knowledge, a novel contribution, in the context of the PGD algorithm. Second, it uses a post-projection splitting of the step to mitigate the breakage of nonlinear constraints. This is done by scaling the component of the update vector (projected gradient) orthogonal to all constraint gradients, \ie in the null constraint variation space. Doing so makes it possible to reduce nonlinear constraint violations and retrieve feasible solutions faster after a constraint breakage is observed. The scaling of this component considers two elements: the variation of an approximation of the Lagrangian gradient and the constraint breakage history.
The convergence of the modified algorithm under specific hypotheses is then demonstrated. Finally, the algorithm is implemented into a topology optimization solver, DFEM  \cite{audet2008dfemwork,navah2021development},  and is then applied to a canonical heat transfer optimization problem in 3D with two different sets of constraints. Using these cases, the approach's performance is compared with the more traditional implementation of the PGD algorithm and other classical optimization methods, such as MMA. The influence of the different tuning parameters of the algorithm on the optimization process is then discussed. Finally, this work discusses the impact of the proposed improvements and future developments.
\bmsection{Projected Gradient Descent Algorithm}
\label{sec::method}
The inertial PGD algorithm solves the following optimization problem:
\begin{align}
   \min_{\bm{\phi}} :  \; & \mathcal{C}(\bm{\phi}), \label{eq::min_cost_function}\\
   \text{s.t.} : \; & f_j(\bm{\phi}) =  a_j \text{ for } j=0,1,2, \ldots, l-1, \label{eq::eq_constraints_function}\\
   : \; & f_j(\bm{\phi}) \leq a_j \text{ for } j=l,l+1, \ldots, m-1, \label{eq::ineq_constraints_function}
\end{align}
where
\begin{align}
    \bm{\phi} &= [\phi_0, \phi_1, \ldots, \phi_{k-1}].
\end{align}

The set of points $\bm{\phi}$ that satisfies the constraints is called the feasible set and is denoted as $\mathcal{Q}$. The inertial PGD algorithm produces an iterative sequence of steps to minimize the cost function $\mathcal{C}(\bm{\phi})$ in the following form:
\begin{align}
    \bm{\Tilde{\phi}}^{n+1}&=\bm{\phi}^n-\Delta\bm{\Tilde{\phi}}^n,\\
    \Delta\bm{\Tilde{\phi}}^n&=\alpha^n  \dd {\mathcal{C}(\bm{\phi}^n)}{\bm{\phi}^n}-\beta^n(\bm{\phi}^n-\bm{\phi}^{n-1}), \label{eq::delta_phi_tilde}\\
    \bm{\phi}^{n+1}&=\text{proj}_{\mathcal{Q}}(\bm{\Tilde{\phi}}^{n+1})=\bm{\phi}^n- \Delta\bm{\phi}^n, \label{eq::phi_nplusone}
\end{align}
where the update vector ($\bm{\phi}^{n}-\text{proj}_{\mathcal{Q}}(\bm{\Tilde{\phi}}^{n+1})$) is denoted as $\Delta \bm{\phi}^n$. Here $\alpha^n$ is the step size, and $\beta^n$ is the inertial step size at the $n$ iteration. The projection onto the feasible set is done by first solving the following sub-optimization to find $\Delta \bm{\phi}^n$ directly, which is then used to find $\bm{\phi}^{n+1}$:
\begin{align}
    \min_{\Delta \bm{\phi}^n} :  \; & \frac{1}{2}\left\Vert\Delta\bm{\Tilde{\phi}}^n-\Delta \bm{\phi}^n \right\Vert^2, \label{eq::min_update} \\
    \text{s.t.}  :  \; &  f_j(\bm{\phi}^n-\Delta \bm{\phi}^n) = a_j \text{ for } j=0,1,2, ..., l, \label{eq::eq_constraints_on_update}\\
     :  \; &  f_j(\bm{\phi}^n-\Delta \bm{\phi}^n) \leq a_j \text{ for } j=l+1,l+2, ..,m.  \label{eq::ineq_constraints_on_update}
\end{align}
Note that in that case, the gradient of this minimization problem simply becomes $\Delta \bm{\phi}^n- \Delta\bm{\Tilde{\phi}}^n$.
In the presence of nonlinear constraints, meaning that if any of the $f_j(\bm{\phi})$ from Equations \eqref{eq::eq_constraints_function} and \eqref{eq::ineq_constraints_function} is a nonlinear function, it is linearized around the current iterate $\bm{\phi}^n$. This is an established method for solving nonlinear optimization problems \cite{torrisi2018projected}. 

The following sections aim to improve the PGD algorithm's applicability to topology optimization problems by enhancing two aspects of the algorithm: first, the treatment of univariate constraints within an active set framework that manipulates constraints in bulk, and second, the incorporation of a post-projection step adjustment to improve nonlinear constraint respect.

\subsection{Application of the active set method}
\label{sec::active_set}
The QP problem that arises from Equations \eqref{eq::min_update}, \eqref{eq::eq_constraints_on_update}, and \eqref{eq::ineq_constraints_on_update}, is solved using an infeasible active set approach. Solving these equations gives the optimal projection of $\bm{\tilde{\phi}}^{n+1}$ onto the feasible set $\mathcal{Q}$ using a linearization of the constraint equations $f_j(\bm{\phi})$ around the current iterate $\bm{\phi}^n$. This projection problem, when solved using an active set, can be separated into two sections: the projection and the active set identification. The projection step on a given active set is presented first, followed by the definition of the active set identification algorithm.

\subsubsection{Projection operation}
The projection on a pre-identified set of constraints in the linearized projection problem involves projecting the point $\bm{\tilde{\phi}}^{n+1}$ onto the intersection of the hyperplanes formed by the active set of linearized constraints to find $\Delta \bm{\phi}^n$. This implies that this iteration treats inequality constraints within the active set as equality constraints. The set of indices of the active constraints is denoted as $J$, and the subspace bound by the constraint part of $J$ is noted $\mathcal{J}$.

The resulting projection of $\bm{\tilde{\phi}}^{n+1}$ onto $\mathcal{Q}$ using the Lagrange multiplier approach is given by:
\begin{align}
    \bm{\phi}^{n+1} &= \bm{\phi}^n - \left(  \Delta\bm{\Tilde{\phi}}^n+ \sum_{j \in J} y_j \dd {f_j(\bm{\phi}^n)}{\bm{\phi}^n}\right),
    \label{eq::updated_point}
\end{align}
where $y_j$ are the Lagrange multipliers associated with the current active set. The Lagrange multipliers are obtained by solving the following matrix system resulting from the linearization of Equations \eqref{eq::eq_constraints_on_update} and \eqref{eq::ineq_constraints_on_update}:
\begin{align}
    A_{{i^*}{j^*}} &= \dd {f_{i^*}(\bm{\phi}^n)}{\bm{\phi}^n} \cdot \dd {f_{j^*}(\bm{\phi}^n)}{\bm{\phi}^n}, \\
    b_{i^*} &= -(a_{i^*} - f_{i^*}(\bm{\phi}^n)) -\dd {f_{i^*}(\bm{\phi}^n)}{\bm{\phi}^n} \cdot   \Delta\bm{\Tilde{\phi}}^n, \\
    \bm{y} &= A^{-1} \bm{b}.
\end{align}

Here, the indices ${i^*}$ and ${j^*}$ refer to the constraints' indices within the active set, not their global indices (\ie the indices of the constraint within Equation \eqref{eq::eq_constraints_function} and \eqref{eq::ineq_constraints_function}). This means that for each component of the vector $\bm{y}$, we must record to which global constraint index it corresponds to. 

In topology optimization problems, it is common to have many univariate constraints in the form of $d \leq \phi_i \leq e$. These constraints are typically maintained using a min--max function clipping \cite{bertsekas2014constrained}. The main advantage of this approach is that the linear problem resulting from the projection of a point onto the intersection of the constraints' hyperplanes remains relatively small since, for most applications, $m \ll k$ in the absence of univariate constraints. However, the clipping approach often results in a loss of optimality in the projection because the Lagrange multipliers associated with other constraints are either adjusted after the projection step to preserve the univariate constraints without considering the optimality of the projection or the Lagrange multipliers are not corrected, leading to none respect of constraints.

To address this issue, the projection step is resolved using a Schur complement approach, which directly incorporates the univariate constraints into the projection step. This method allows for an accurate evaluation of the contribution of univariate constraints to the update vector $\Delta \bm{\phi}$. The active constraints are divided into two blocks: global constraints (\ie those that are not univariate) and univariate constraints. The number of active global constraints is denoted by $m^*$ while the total size of the active constraint set is denoted $\hat{m}$, with $\hat{m} \leq m$. Using this separation, the matrix used to calculate the Lagrange multipliers takes the following form; note that hereafter, for simplicity $f_j(\bm{\phi}^n) \equiv f_j $:

\begin{align}
    M&=
    \begin{bmatrix}
        A & B^t \\
        B & D 
    \end{bmatrix}\label{eq::schur_complement_1} \\
    A_{{i^*}{j^*}} &= \dd {f_{i^*}}{\bm{\phi}^n} \cdot \dd {f_{j^*}}{\bm{\phi}^n} \quad \text{for } {i^*} < m^* \text{ and } {j^*} < m^*\label{eq::schur_complement_2} \\
    B_{({i^*}-m^*){j^*}} &= \frac{d f_{i^*}}{d \phi^n_{\hat{k}}} \dd {f_{j^*}}{\phi^n_{\hat{k}}}  \quad \text{for } {i^*} \geq m^* \text{ and } {j^*} < m^* \label{eq::schur_complement_3}\\
    D_{({i^*}-m^*)({j^*}-m^*)} &= 
    \begin{cases}
        \left(\frac{d f_{i^*}}{d \phi^n_{\hat{k}}}\right)^2, & \text{if } {i^*} = {j^*} \\
        0, & \text{otherwise}
    \end{cases} \quad \text{for } {i^*} \geq m^* \text{ and } {j^*} \geq m^* \label{eq::schur_complement_4}\\
    \bm{b}&=\begin{bmatrix}
        \bm{b_1} \\
        \bm{b_2}
    \end{bmatrix}= 
        \begin{cases}
       -(a_{i^*} - f_{i^*}) - \dd {f_{i^*}}{\bm{\phi}^n} \cdot   \Delta\bm{\Tilde{\phi}}^n, & \text{if } {i^*} < m^* \\
        -(a_{i^*} - f_{i^*}) - \frac{d f_{i^*}}{d \phi^n_{\hat{k}}}  \Delta\Tilde{\phi}^n_{\hat{k}}, & \text{otherwise}
    \end{cases} \label{eq::schur_complement_5} \\
    \bm{y}&=\begin{bmatrix}
        \bm{y_1} \\
        \bm{y_2}
    \end{bmatrix}= {M}^{-1} \bm{b}
    \label{eq::schur_complement_6}
\end{align}

In Equations \eqref{eq::schur_complement_1} to \eqref{eq::schur_complement_6}, the indices $\hat{k}$ refer to the components of $\bm{\phi}^n$ to which the univariate constraints correspond. Here, the block $D$ is a simple diagonal matrix, and in the case of our univariate constraint in the form  $d \leq \phi_i \leq e$, it takes the form of an identity matrix of size $\hat{m}-m^*$. This significantly simplifies the resolution, making the assembly and resolution of this system a much more manageable task using a Schur complement:
\begin{align}
    \bm{y_1} &= (A-BD^{-1} B^t)^{-1} (\bm{b_1}-B^tD^{-1}\bm{b_2}),\\
    \bm{y_2} &= D^{-1}(\bm{b_2}-B\bm{y_1}).
\end{align}
Indeed, since $D$ is diagonal, the resolution only implies the resolution of an $m^*$ by $m^*$ system of linear equations. Note that assembling this matrix system does not require heavy communication in a parallel distributed software architecture since most communication involves distributing the results of vector dot products, which are highly optimized operations. The other necessary operations are all local to each processor. Without the Schur complement, the resolution of the linear system arising in the projection problem would be impractical for topology optimization problems. Additionally, since this projection is the solution to a linear equation, the point obtained is the optimal point projection on this intersection of plans.

\subsubsection{Active set definition}
An iterative approach is used to identify the active set of constraints, where constraints are added and removed if the Lagrange multipliers are negative. Each sub-iteration of the active set definition is indexed by $o$. If some constraints from the active set are removed, the Lagrange multipliers are reevaluated before defining the update. However, the large number of univariate constraints requires bulk manipulation of constraints to be numerically efficient. As such, the algorithm steps are modified to account  for the bulk manipulation of constraints, resulting in the following steps:
\begin{enumerate}
    \item Add all the equality constraints into the initial active constraints set $\mathcal{J}^0$.
    \item Initialize the update vector $\Delta \bm{\phi}^{n,o}$ with $  \Delta\bm{\Tilde{\phi}}^n$ (Equation \eqref{eq::delta_phi_tilde}). 
    \item Add all the inequality constraints (including univariate constraints) that would be broken along the current update vector $\Delta \bm{\phi}^{n,o}$ to the previous active set $\mathcal{J}^{o-1}$ to initialize the current active set $\tilde{\mathcal{J}^{o}}$. This is done by evaluating a linearization of Equation  \eqref{eq::ineq_constraints_on_update}:
    \begin{align}
     f_j-\Delta \bm{\phi}^n \cdot \dd{f_j}{\bm{\phi}^n} \leq a_j \text{ for } j=l+1,l+2, ..,m. 
    \end{align}
    \item Evaluate the size of the active set of constraints. If its size is superior to the size of $\bm{\phi}^n$ (if $\hat{m} > k$), relax the initial guess $\Delta\bm{\Tilde{\phi}}^n$ by a factor $\zeta \in (0,1)$. In practice, this is done by multiplying both $\alpha^n$ and $\beta^n$ by $\zeta$ and reevaluating $\Delta\bm{\Tilde{\phi}}^n$. The value $\zeta =0.5$ was found to be adequate. Return to step 1 with the new initial guess and clear the active set. 
    \item Compute the projection of $\Delta\bm{\Tilde{\phi}}^n$ on the current active set $\tilde{\mathcal{J}}^o$ by solving Equations \eqref{eq::min_update},\eqref{eq::eq_constraints_on_update}, and \eqref{eq::ineq_constraints_on_update}, using Equations \eqref{eq::schur_complement_1} to \eqref{eq::schur_complement_6}. 
      and update $\Delta \bm{\phi}^{n,o}$, using 
    \begin{align}
        \Delta \bm{\phi}^{n,o} = \Delta\bm{\Tilde{\phi}}^n+ \sum_{j \in J^o} y_j \dd{f_j(\bm{\phi}^n)}{\bm{\phi}^n},
        \label{eq::update_loop_o}
    \end{align}
    and the newly evaluated value for the Lagrange multipliers $\bm{y}$.
    \item Check if the cost function increased, satisfying:
    \begin{align}
       \Vert \Delta\tilde{\bm{\phi}}^n - \Delta\bm{\phi}^{n,o} \Vert \leq \Vert \Delta\tilde{\bm{\phi}}^n - \Delta\bm{\phi}^{n,o+1} \Vert. 
       \label{eq::cost_increase_1}
    \end{align}
    If it decreases,the algorithm cannot guarantee convergence; as such, follow these sub-steps, otherwise go to the next step:
    \begin{enumerate}
        \item Reincorporate constraints from the stored copy of  $\mathcal{J}^{o-1}$ that were removed into $\tilde{\mathcal{J}^{o}}$. 
        \item Repeat the same procedure as in step 5 and project $ \Delta\tilde{\bm{\phi}}^n $ on the current active set $\tilde{\mathcal{J}^{o}}$ and obtain the new estimate for $\Delta \bm{\phi}^{n,o} $.
        \item Check if this new iterate does respect Inequality \eqref{eq::cost_increase_1}.
        \item \begin{enumerate}
        \item If the inequality is satisfied, remove the constraint with the minimal Lagrange multiplier value from $\tilde{\mathcal{J}^{o}}$ (assuming it is negative). Simultaneously, remove it from the stored copy of $\mathcal{J}^{o-1}$. Then go back to step 5 with the updated active set $\tilde{\mathcal{J}_{o}}$.
        \item If the inequality is not satisfied, check which constraints that were removed from the stored copy of $\mathcal{J}^{o-1}$ (relative to the original previous set) are now broken. Reintroduce the most binding constraints to the active set $\tilde{\mathcal{J}_{o}}$ and the stored copy of $\mathcal{J}^{o-1}$. The most binding constraint is the one with the maximum value of:
        \begin{align}
             \frac{f_j-\Delta \bm{\phi}^{n,o} \cdot \dd{f_j}{\bm{\phi}^n}-a_j}{\Vert \dd{f_j}{\bm{\phi}^n} \Vert}.
        \end{align}
        \end{enumerate}
        
    \end{enumerate}
    \item  Remove all inequality constraints within the active set having a negative Lagrange multiplier. This step is split into two conditions:
        \begin{enumerate}
            \item If newly added constraints $(\tilde{\mathcal{J}}^o \setminus \mathcal{J}^{o-1})$ have negative Lagrange multipliers, remove all of them from the active set and go back to step 5 with the updated active set.
            \item Otherwise, if any remaining constraints have a negative Lagrange multiplier, remove them and return to step 5 with the updated active set.
        \end{enumerate}
    If constraints are removed from the copy of $\mathcal{J}^{o-1}$, they are stored in case it is required to reincorporate them in $\tilde{\mathcal{J}}^o$. If no constraint was removed from the active set, indicating that the KKT criteria on the active set were respected, the active set of the current iterate $o$ is found and is noted $\mathcal{J}^o$.
    \item Finish iteration $o$ by evaluating $\Delta \bm{\phi}^{n,o+1}$ with the current active set and  Equation \eqref{eq::update_loop_o}, and define the final active set ($\mathcal{J}^o$) of this iteration ($\mathcal{J}^o=\tilde{\mathcal{J}}^o$).
    \item If, during steps 3 to 7, a constraint was either added or removed from the active set, go back to step 3 with the newly obtained $\Delta \bm{\phi}^{n,o+1}$; otherwise, exit the algorithm and use $\Delta \bm{\phi}^n=\Delta \bm{\phi}^{n,o+1}$ as the solution of the projection.
\end{enumerate}
\begin{sidewaysfigure}[!htpb]
	\centering
	\includegraphics[width=0.9\textwidth]{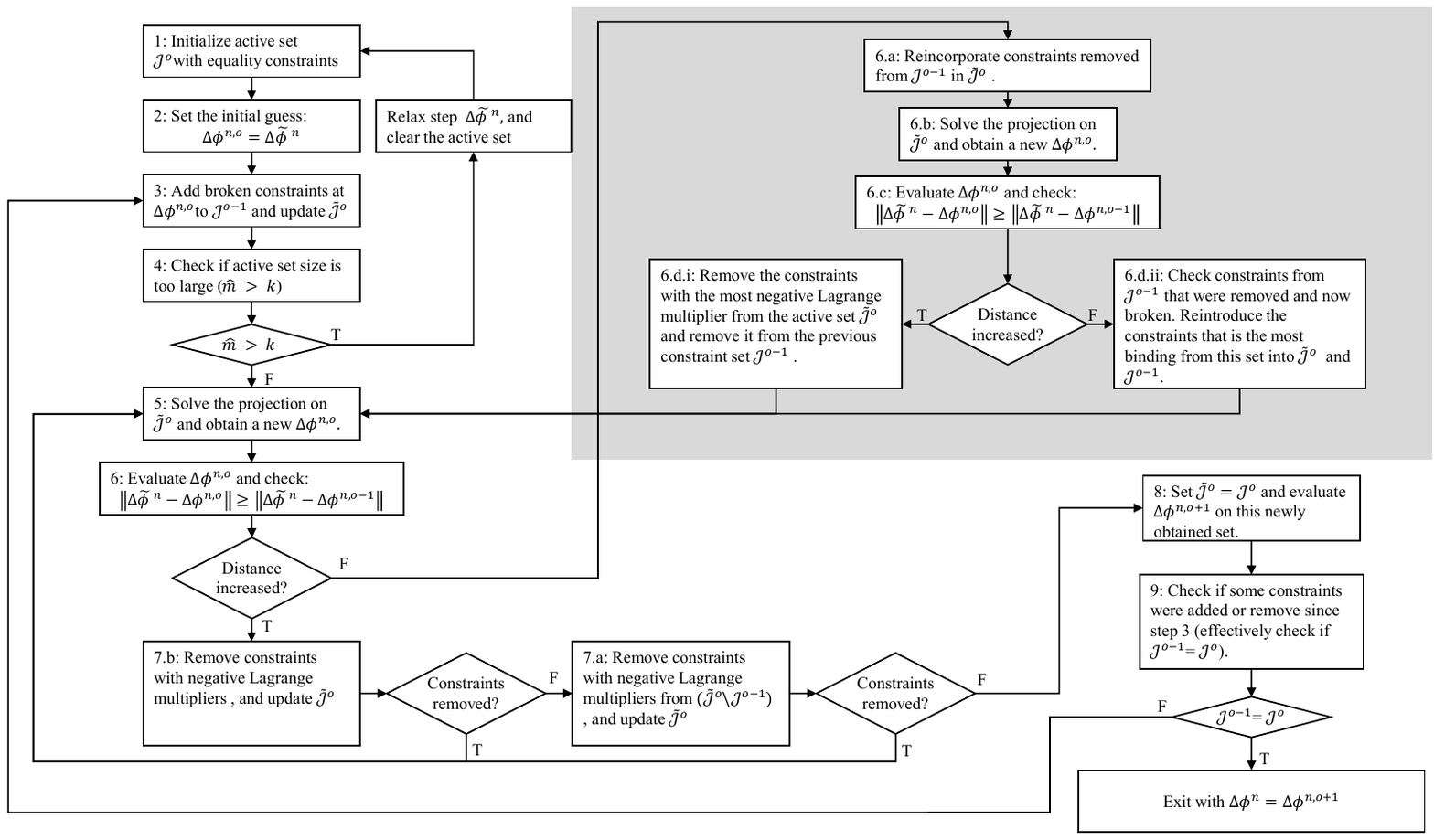}
	\caption{Active set flowchart diagram. If the path is marked by a T, the condition or question is true, and this path must be followed; otherwise, follow the path mark with F. The block of operations in the light grey zone is rarely taken when the algorithm cannot guarantee progress (See appendix \ref{appendix::frequency_6_6c_fails} for more details on the frequency).}
\label{fig::active_set_flowchart}
\end{sidewaysfigure}

To facilitate the understanding, Figure \ref{fig::active_set_flowchart} presents the workflow of the active set algorithm. The latter provides a projection of the point $\bm{\tilde{\phi}}^{n+1}$ onto the feasible set $\mathcal{Q}$ that minimizes Equation \eqref{eq::min_update}. The algorithm further manipulates constraints in bulk as much as possible to accelerate the definition of the final active set.
However, adding and subsequently removing multiple constraints from the active set can lead to oscillation, where constraints are repeatedly included and excluded from the active set. To eliminate oscillation and ensure convergence, the Lagrange multipliers are reevaluated using only the constraints with positive Lagrange multiplier values \cite{wong2011active} (loop between steps 5 and 7). Finally, suppose the algorithm cannot guarantee convergence of a given step, using the condition verified in step 6 and step 6c. In that case, it reverts to single constraint manipulation in the removal and reintroduction process to ensure convergence (step 6d). The remainder of the section demonstrates the convergence of this modified algorithm.

Let's introduce the KKT criteria for the point projection problem from Equations \eqref{eq::min_update}, \eqref{eq::eq_constraints_on_update}, and \eqref{eq::ineq_constraints_on_update}. Let $\mathcal{L^*}(\bm{\phi}^n)$ denote the Lagrangian of the projection problem, then optimality conditions are:

\begin{align}
    \dd {\mathcal{L^*}(\bm{\phi}^n)}{\bm{\phi}^n} &= (\Delta \bm{\phi}^n- \Delta\bm{\Tilde{\phi}}^n)  + \sum_{j\in J} y_j \dd {f_j(\bm{\phi}^n)}{\bm{\phi}^n} = \bm{0}, \label{eq::lagrangien_gradient_proj} \\
     f_j-\Delta \bm{\phi}^n \cdot \dd{f_j}{\bm{\phi}^n}  &= a_j \text{ for } j=0,1,2, \ldots, l, \label{eq::eq_constraint_kkt}\\
     f_j-\Delta \bm{\phi}^n \cdot \dd{f_j}{\bm{\phi}^n}  &\leq a_j \text{ for } j=l+1, l+2, \ldots, m, \label{eq::ineq_constraint_kkt} \\
    y_j &\geq 0 \text{ for } j=l+1, l+2, \ldots, m.
    \label{eq::dual_feasibility}
    \\
    (f_j - a_j) y_j &= 0 \text{ for } j=l+1, l+2, \ldots, m.
    \label{eq::complementary_slackness}
\end{align}

To show the convergence, let's introduce the concept of a self-compatible set of constraints. A self-compatible set of constraints is a collection of constraints that, when enforced together, ensures the satisfaction of all KKT conditions: feasibility, non-negative Lagrange multipliers, complementary slackness, and stationarity within this subset. Using this concept, the proof of convergence will demonstrate that the active set will be more constraining at each iteration $o+1$ than at the previous iteration ($\Vert \Delta\tilde{\bm{\phi}}^{n} - \Delta\bm{\phi}^{n,o} \Vert \leq \Vert \Delta\tilde{\bm{\phi}}^{n} - \Delta\bm{\phi}^{n,o+1} \Vert$), increasing the cost function while ensuring that the active set is self-compatible. This will ensure that all constraints are satisfied and that all of the KKT criteria are met at convergence, assuming that all constraints within the active set are linearly independent and that $\mathcal{Q}$ is feasible. 

Note that the algorithm guarantees that the active set obtained at the end of each iteration is self-compatible due to the exit condition in step 7, where the iteration only proceeds if no constraint is removed from the active set. This condition implies that all constraints in the currently active set have non-negative Lagrange multipliers, are satisfied, and the projection of $\Delta\tilde{\bm{\phi}}^{n}$ is stationary, establishing self-compatibility. Such sets possess essential properties that facilitate convergence analysis within our algorithm, which are formalized in the following lemmas.

First, Lemma \ref{lemma::constraint_violation} provides a criterion to detect when constraints within a self-compatible set are violated based on changes in the cost function. Second, Lemma \ref{lemma::cost_increase} establishes that the cost function projection on a self-compatible set must be less than the cost function at the optimal solution, should this projection violate additional constraints. These lemmas are satisfied, provided that the cost function of the sub-optimization problem is convex.

\begin{lemma}
If the cost function value $C^{o}$ for the optimal solution of a constrained optimization problem using a set of constraints $\tilde{\mathcal{J}}^o$ results in a decrease in the cost function relative to $C^{o-1}$, the cost function for the optimal solution on a self-compatible set $\mathcal{J}^{o-1}$ ($C^{o} < C^{o-1}$), then a least one of the constraints in $\mathcal{J}^{o-1}$ is violated by the optimal solution obtained with the set of constraints $\tilde{\mathcal{J}}^o$. 
\label{lemma::constraint_violation}
\end{lemma}
\begin{proof}
Let $\mathcal{J}^{o-1}$ be a self-compatible set of constraints, the optimal solution of the minimization problem on the subset of constraint $\mathcal{J}^{o-1}$ (here the projection of $\Delta \tilde{\bm{\phi}}^n$ onto $\mathcal{J}^{o-1}$) is optimal for this subset (due to self-compatibility) with a value $C^{o-1}$. Then, suppose that the optimal solution on the subset $\tilde{\mathcal{J}}^o$ has a cost $C^{o}$ smaller than $C^{o-1}$. In that case, it indicates that $\mathcal{J}^{o-1}$ is no longer feasible by the optimal solution onto $\tilde{\mathcal{J}}^o$ since no feasible point defined by the feasible region $\mathcal{J}^{o-1}$ has a smaller cost than  $C^{o-1}$. Thus, at least one constraint in $\mathcal{J}^{o-1}$ must be broken by the optimal solution onto the subset $\tilde{\mathcal{J}}^o$.
\end{proof}

\begin{lemma}
If the optimal solution of a constrained optimization problem using a self-compatible set of constraints $\mathcal{J}^{o-1}$ results in the breaking of additional constraints $\mathcal{J}_b$ (\ie constraints outside of $\mathcal{J}^{o-1}$ are violated), then the optimal solution over the set of constraints $\mathcal{J}^{o-1} \cup \mathcal{J}_b$ must yield a cost function value higher or equal than the optimal solution with the constraints set $\mathcal{J}^{o-1}$ alone.
\label{lemma::cost_increase}
\end{lemma}
\begin{proof}
The optimal solution for the constrained optimization problem on a self-compatible set $\mathcal{J}^{o-1}$ minimizes the cost function for that specific set. If additional constraints outside of $\mathcal{J}^{o-1}$ are broken, incorporating these constraints into the active set creates a more constrained feasible region, necessitating moving away from the previous optimal point to satisfy all constraints. By the Pythagorean theorem, moving to a more constrained set requires an adjustment from $\Delta \tilde{\bm{\phi}}^n$, resulting in a higher cost than the projection onto $\mathcal{J}^{o-1}$ alone. Thus, the cost increases when additional broken constraints are included.
\end{proof}

With these properties established, Theorem \ref{theo::active_set_proof} can now be introduced to demonstrate the convergence of the active set algorithm.

\begin{theorem}
The algorithm of Section \ref{sec::active_set} presented in Figure \ref{fig::active_set_flowchart} converges to a stable active set that satisfies the KKT criteria of Equations \eqref{eq::lagrangien_gradient_proj} to \eqref{eq::complementary_slackness}
\label{theo::active_set_proof}
\end{theorem}
\begin{proof}
To ensure convergence, two properties of the algorithm must be satisfied. First, the algorithm is monotonic, meaning each set obtained at the end of each iteration is new. Second, at convergence, all properties of the KKT criteria (Equations \eqref{eq::lagrangien_gradient_proj} to \eqref{eq::complementary_slackness}) are satisfied.

Since each iteration produces a self-compatible set due to the exit condition in step 7, according to Lemma \ref{lemma::cost_increase}, if the projection of $ \Delta \tilde{\bm{\phi}}^n $ onto a self-compatible set $ \mathcal{J}^{o-1} $ leaves some constraints violated ($\mathcal{J}_b$), then optimal projection onto an expanded set $ \mathcal{J}^{o-1} \cup \mathcal{J}_b =  \tilde{\mathcal{J}}^o$ will yield a higher cost function. This result justifies the algorithm’s design to check for an increase in the cost function in step 6, as any decrease would indicate that the previously self-compatible set $ \mathcal{J}^{o-1} $ is no longer feasible under the new constraints. This can then be use to show monotonicity.

Let $F(\Delta\bm{\phi}^{n,o})$ represent a merit function, chosen here to be directly linked to the cost function $F(\Delta\bm{\phi}^{n,o})=\left\Vert\Delta\bm{\Tilde{\phi}}^n-\Delta \bm{\phi}^{n,o} \right\Vert$. The algorithm is monotonic if we can show that:
\begin{align}
       F(\Delta\bm{\phi}^{n,o+1})\geq F(\Delta\bm{\phi}^{n,o}).
       \label{eq::cost_increase_2}
\end{align}
The algorithm, through step 6, explicitly checks for monotonicity. Two scenarios arise: the Inequation \eqref{eq::cost_increase_2} is satisfied or not.

If an increase in the cost function is observed, the algorithm simply moves to the next step, trying to remove constraints with negative Lagrange multipliers to ensure optimality.

When step 6 fails to satisfy the cost increase condition, the algorithm falls back to single-constraint manipulation to re-establish feasible progress. This fallback to individual constraint handling mirrors classical active set methods, ensuring that each constraint addition or removal is carefully validated in sequence. By adapting in this way, the algorithm effectively manages complex interactions between constraints, especially when bulk operations fail to ensure progress toward the optimal solution.

If a decrease in the cost function is observed, according to Lemma \ref{lemma::constraint_violation}, it implies that at least one constraint in $\mathcal{J}^{o-1}$ is violated. The algorithm then reverts the active set to reinclude all constraints from a copy of the previous active set $\mathcal{J}^{o-1} $. By doing so, it initially reverts to a set of constraints that guarantee an increase in the cost function. Since, in this case, $\mathcal{J}^{o-1} \subset \tilde{\mathcal{J}}^o$, then by Pythagorean inequality, the projection over $\tilde{\mathcal{J}}^o$ will be further away than the projection over $\mathcal{J}^{o-1} $. This ensures that condition 6.c will be satisfied initially. Then, at step 6.d, the constraint with the most negative Lagrange multiplier is removed from the active set $\tilde{\mathcal{J}}^o$, as well as the copy of the previous active set  $\mathcal{J}^{o-1} $. Removing a single constraint with a negative Lagrange multiplier from $\tilde{\mathcal{J}}^o$ guarantees, by the complementary slackness (Equation \eqref{eq::complementary_slackness}), that it will be satisfied by the projection on the updated set $\tilde{\mathcal{J}}^o$. This then still ensures that condition 6.c will be satisfied. Removing the constraint from the copy $\mathcal{J}^{o-1} $ also prevents reintegrating it again in the active set, which could lead to a cycle of repeatedly removing and reading constraints.

However, if the algorithm ends up removing multiple constraints from the copy of $\mathcal{J}^{o-1}$ due to repeatedly not satisfying the condition of step 6, then reintroducing these constraints does not guarantee that condition 6.c will be met. In this scenario, Lemma \ref{lemma::constraint_violation} confirms that at least one of the constraints removed from the copy of $\mathcal{J}^{o-1}$ is now broken, signifying that the current projection is infeasible with respect to this self-compatible set. The algorithm then examines all constraints removed from the copy of $\mathcal{J}^{o-1}$ and reintroduces the most binding constraint that is now broken in $\tilde{\mathcal{J}}^o$. After updating $\tilde{\mathcal{J}}^o$ with this constraint, the algorithm re-evaluates the projection at step 5.

If a decrease in cost is observed again, the algorithm proceeds by reintroducing constraints one by one from the copy of $\mathcal{J}^{o-1}$. This selective reintroduction of constraints effectively reduces the algorithm to a classical active set method, where each constraint is manipulated individually to validate feasibility. This fallback ensures that the algorithm can progress even when bulk constraint operations fail by verifying each constraint addition or removal in sequence.

The convergence of such single-constraint manipulation algorithms is well-documented in the literature, as shown in the works of Wong and Nocedal \textit{et al.} \cite{wong2011active, nocedal1999numerical}. The algorithm will then continue iterating between steps 5 to 7, adding or removing constraints one at a time whenever bulk manipulation fails to meet the condition of step 6. This process continues until all constraints in $\tilde{\mathcal{J}}^o$ have positive Lagrange multipliers (satisfying the exit condition of step 7) and Inequality \eqref{eq::cost_increase_2} is satisfied.

This guarantees monotonicity between iterations of the algorithm and ensures convergence to a stable active set.

Next, the algorithm guarantees that all KKT conditions of Equations \eqref{eq::lagrangien_gradient_proj} to \eqref{eq::complementary_slackness} are satisfied. The exit criteria enforce that a new iteration begins whenever a constraint is added or removed from the active set. The algorithm continues until the projection on the current active set satisfies all constraints without necessitating further additions or removals. The final exit condition is achieved when the following KKT criteria are met at the projection point:

\begin{enumerate}
    \item Stationarity: The optimality condition of the projection inherently satisfies stationarity, thereby fulfilling Equation \eqref{eq::lagrangien_gradient_proj}.

    \item Feasibility: Upon exiting, all equality constraints are satisfied (Equation \eqref{eq::eq_constraint_kkt}), and all inequality constraints are met without violation (Inequation \eqref{eq::ineq_constraint_kkt}). This ensures that the solution lies within the feasible region defined by $\mathcal{Q}$. Otherwise, the algorithm would have added constraints in the last evaluation of step 3, forcing a new iteration.
    
    \item Dual feasibility: The algorithm ensures that all Lagrange multipliers are non-negative at the final projection, thereby meeting the requirement of Inequation \eqref{eq::dual_feasibility}. This is due to the exit condition of step 7, where the algorithm removes constraints with a negative Lagrange multiplier.

    \item Complementary slackness: The final active set guarantees that each constraint is either binding or has a positive Lagrange multiplier, ensuring that Equation \eqref{eq::complementary_slackness} is satisfied.
\end{enumerate}

Thus, the algorithm converges monotonically to a stable active set that satisfies all KKT conditions.
\end{proof}

Although the algorithm performs bulk manipulation of constraints during both the addition and removal stages, it maintains monotonicity by reverting to single-constraint adjustments whenever bulk manipulation fails to satisfy step 6 conditions. This adaptive strategy ensures the satisfaction of the exit condition in step 7, thus guaranteeing convergence. By prioritizing bulk manipulation initially, the algorithm significantly reduces the number of evaluations required to reach the optimal solution for the point projection problem. In practice, the conditions in steps 6 and 6c are rarely not respected, making most operations using bulk manipulation of constraints for the projection problems. Appendix \ref{appendix::frequency_6_6c_fails} presents some results for the frequency at which steps 6 and 6c fail. The efficiency gain in the handling of univariate constraints is sufficient to enable the direct integration of univariate constraints in the projection step. Finally, note that this formulation of the algorithm can still get stuck in the definition of the active set if too many constraints are broken at the current solution ($\bm{\phi}^n$), in which case the algorithm will try to reduce the step size indefinitely (due to step 4). This could be prevented by limiting the number of constraints added and only selecting the most broken constraints, similar to what is done in the work of Schittkowski \cite{schittkowski2009active}.

\subsection{Post--projection step adjustment}
\label{sec::consideration_kkt}
Generally, the PGD algorithm does not use the KKT criteria or the feasibility of non-linear constraints in the definition of step size ($\alpha_n$). This means that the impact of the constraints on the step size is not usually considered. This section presents an approach to improve this aspect of the PGD algorithm by considering the effect of the constraints on the step size through the KKT criteria and the constraints satisfaction. First, let's introduce the KKT for the optimization problem laid out by Equations \eqref{eq::min_cost_function}, \eqref{eq::eq_constraints_function}, and \eqref{eq::ineq_constraints_function}, the Lagrangian take the form :
\begin{align}
    \mathcal{L}(\bm{\phi}) = \mathcal{C}(\bm{\phi}) + \sum_{j=0}^{m} y_j f_j(\bm{\phi}).
    \label{eq::lagrangian_optimization_problem}
\end{align}
From Equation \eqref{eq::lagrangian_optimization_problem}, the following KKT optimality criteria follow:
\begin{align}
    \dd {\mathcal{L}(\bm{\phi}^n)}{\bm{\phi}^n} &= \dd {\mathcal{C}(\bm{\phi}^n)}{\bm{\phi}^n}  + \sum_{j=0}^{m} y_j\dd {f_j(\bm{\phi}^n)}{\bm{\phi}^n}  = \bm{0}, \label{eq::lagrangien_gradient} \\
    f_j(\bm{\phi}) &= a_j \text{ for } j=0,1,2, \ldots, l, \\
    f_j(\bm{\phi}) &\leq a_j \text{ for } j=l+1, l+2, \ldots, m, \\
     y_j &\geq 0 \text{ for } j=l+1, l+2, \ldots, m.
    \label{eq::dual_feasibility0}
    \\
    (f_j(\bm{\phi}) - a_j) y_j &= 0 \text{ for } j=l+1, l+2, \ldots, m.
\end{align}
From the first criterion (Equation \eqref{eq::lagrangien_gradient}), the gradient of the Lagrangian should tend toward zero at local minima. This information can be used to adjust the step length after projecting $\bm{\tilde{\phi}}^{n+1}$ onto $\mathcal{Q}$. To do this, we approximate $\dd {\mathcal{L}(\bm{\phi}^n)}{\bm{\phi}^n}$ at the current iteration, assuming that the active set of constraints remains unchanged relative to a change in $\bm{\phi}$ and that all constraints are binding (meaning that the current iterate $\bm{\phi}^n$ is at the limit of the constraints in the active set):
\begin{align}
     \dd {\mathcal{L}(\bm{\phi}^n)}{\bm{\phi}^n} \approx \frac{\Delta \bm{\phi}^n}{\alpha^n}.
\end{align}
This approximation does not consider the effect of the inertial component of the update vector ($\beta^n=0$ in Equation \eqref{eq::delta_phi_tilde}). This implies that the larger the inertial component ($\beta^n(\bm{\phi}^n -\bm{\phi}^{n-1})$), relative to the size of the gradient portion ($\alpha^n\dd {\mathcal{C}(\bm{\phi}^n)}{\bm{\phi}^n} $), the worst this approximation will become. Despite potential inaccuracies, the variation of this approximation of the Lagrangian gradient can still provide useful information, such as the variation of the constraints between iterations. This approximation of the Lagrangian can then be used to define a post-projection step length adjustment by a factor $\gamma^n$. Applying this correction factor results in the following modification of the update (Equation \eqref{eq::phi_nplusone}):
\begin{align}
    \bm{\phi}^{n+1} = \bm{\phi}^n - \frac{\gamma^n}{\alpha^n} \Delta \bm{\phi}^n,
\end{align}
where $\gamma^n$ is a step size modification factor scaled relative to the step size $\alpha^n$. Therefore, methods such as the Barzilai-Borwein method can be used to evaluate $\gamma^n$ based on the gradient variation of the Lagrangian. Defining the step in such a way accounts for the variations of the Lagrangian caused by the impact of the constraints, which can improve convergence. Further details on the exact formulation used to define $\gamma^n$ based on the Lagrangian gradient approximation are given in Section \ref{sec::convergence_analysis}. 

However, applying $\gamma^n$ directly to the step $\Delta \bm{\phi}^n$ risks breaking the constraints if $\frac{\gamma^n}{\alpha^n}>1$. To avoid this issue, the step modification is only applied to the null constraint variation space of the update vector $\Delta \bm{\phi}^n$, if $\frac{\gamma^n}{\alpha^n}>1$. To do this, $\Delta \bm{\phi}^n$ is split into two components: the component orthogonal to the constraint gradients ($\dd {f_j(\bm{\phi}^n)}{\bm{\phi}^n}$), noted as $\Delta \bm{\phi}^n_\perp$, and the remaining portion of the vector, noted as $\Delta \bm{\phi}^n_\parallel$. This decomposition is performed using a Gram--Schmidt process \cite{strang2012linear}. In practice, the decomposition can be performed by either considering only the constraints within the active set or all constraints. Using the former approach leads to a more aggressive update, which can occasionally break some of the constraints, although these will be identified and corrected in subsequent iterations. 
\begin{align}
    \bm{\phi}^{n+1} = \bm{\phi}^n - \min\left(1,\frac{\gamma^n}{\alpha^n}\right) \Delta \bm{\phi}^n_\parallel - \frac{\gamma^n}{\alpha^n} \Delta \bm{\phi}^n_\perp.
\end{align}

This update prevents violating linear constraints, but not nonlinear ones. To mitigate this issue, an additional relaxation parameter $\mu \in (0,1]$ can be added to relax the orthogonal update when constraints are broken. The relaxation parameter $\mu$ is applied to $\Delta \bm{\phi}^n_\perp$ with an exponent equal to the consecutive number of steps where at least one constraint was broken. Additionally, in the presence of broken constraints, to correct the error, $\gamma$ is not applied to the constraint variation component ($\Delta \bm{\phi}^n_\parallel$). This gives the final form of the update:
\begin{align}
    \bm{\phi}^{n+1} =
    \begin{cases}
       \bm{\phi}^n - \min(1,\frac{\gamma^n}{\alpha^n}) \Delta \bm{\phi}^n_\parallel - \frac{ (\mu)^{^h}\gamma^n}{\alpha^n} \Delta \bm{\phi}^n_\perp, & \text{if }  \bm{\phi}^n \in \mathcal{Q}, \\
       \bm{\phi}^n - \Delta \bm{\phi}^n_\parallel - \frac{ (\mu)^{^h} \gamma^n}{\alpha^n} \Delta \bm{\phi}^n_\perp,  & \text{otherwise},
    \end{cases}
    \label{eq::update_with_mu}
\end{align}
where $h$ is the number of consecutive steps where at least one of the constraints was broken. Adding this parameter causes the optimizer to give more weight to constraint correction relative to optimization if constraints remain broken over multiple successive steps. When the constraint is once again satisfied, $h$ can either be reset to  0 directly, or decreased smoothly back to 0,  \eg by subtracting 1 following each optimization step where constraints are respected. The latter approach is recommended and used in the rest of this paper since it is more conservative and helps prevent the breaking of constraints once they are once again satisfied. A value of $\mu=0.95$ is sufficient for most cases, as $(\mu)^{^h}$ decreases sufficiently fast when a constraint is broken for consecutive iterations without relaxing the optimization process. In practice, a tolerance $\epsilon_i$ is added to the constraints before defining it as broken to avoid needlessly reducing the step size when the solution is on the boundary of $\mathcal{Q}$. The tolerance $\epsilon_i$ is a user-defined parameter that reflects the user's tolerance for each constraint.
Additionally, two approaches can be used to avoid breaking univariate constraints due to the modification of the step with $\gamma^n$. First, a traditional min--max clipping approach can be applied to the new step obtained. However, since the component of the vector that is scaled by $\gamma^n$ is in the null constraint variation space, the constraint modification due to the clipping is usually negligible. The second alternative corresponds to simply reapplying the projection step to the point obtained by the new step with the Lagrangian estimation scaled by $\gamma^n$ to the feasible set $\mathcal{Q}$ using the same projection algorithm previously used and described in Section \ref{sec::active_set}. The latter option increases the update evaluation's cost since it requires another projection step's evaluation. This approach completely removes the need to use min-max clipping and avoids the issues that can come with it. However, this approach was not used here since it did not significantly modify the results, as the min--max clipping effect on the constraints was negligible, as mentioned before.

\subsection{Convergence analysis}
\label{sec::convergence_analysis}
Even though the convergence properties of the Inertial Projected Gradient Descent are well known \cite{nishioka2023inertial}, this section aims to demonstrate that these properties still hold when the step is modified, as discussed in Section \ref{sec::consideration_kkt}. The convergence analysis will further serve to provide a formal definition for the inertial parameter $\alpha ^n$, $\gamma^n$, and $\beta^n$. Assuming that the gradient of the cost function is Lipschitz continuous and that the solution remains within the feasible domain($\bm{\phi}^n \in \mathcal{Q}$), the following inequality holds:
\begin{equation}
    \begin{split}
    \mathcal{C}(\bm{\phi^{n+1}})\leq &\;\mathcal{C}(\bm{\phi^{n}})- \dd {\mathcal{C}(\bm{\phi}^n)}{\bm{\phi}^n} \cdot \left( \min\left(1,\frac{\gamma^n}{\alpha^n}\right)\Delta\bm{\phi^{n}}_\parallel +\frac{\gamma^n}{\alpha^n}\Delta\bm{\phi^{n}}_\perp\right)\\ &+\frac{L}{2} \left\Vert- \min\left(1,\frac{\gamma^n}{\alpha^n}\right)\Delta\bm{\phi^{n}}_\parallel-\frac{\gamma^n}{\alpha^n}\Delta\bm{\phi^{n}}_\perp\right\Vert^2.
    \end{split}
\end{equation}

Here, $L$ is the Lipschitz constant of the gradient of the cost function. From this, for $\mathcal{C}(\bm{\phi^{n+1}})\leq \mathcal{C}(\bm{\phi^{n}})$ to hold the following inequation must also be true:
\begin{equation}
    \begin{split}
     -\dd {\mathcal{C}(\bm{\phi}^n)}{\bm{\phi}^n} \cdot \left(\min(1,\frac{\gamma^n}{\alpha^n})\Delta\bm{\phi^{n}}_\parallel +\frac{\gamma^n}{\alpha^n}\Delta\bm{\phi^{n}}_\perp\right) +\frac{L}{2} \left\Vert-\Delta\bm{\phi^{n}}_\parallel-\frac{\gamma^n}{\alpha^n}\Delta\bm{\phi^{n}}_\perp\right\Vert^2 \leq 0.
    \end{split}
\end{equation}

Substituting $\Delta\bm{\phi^{n}}_\parallel$ and $\Delta\bm{\phi^{n}}_\perp$ by their definition as projections on two orthogonal subspaces,
 \begin{align}
 \Delta\bm{\phi^{n}}_\parallel&=\lambda_\parallel \left(\alpha^n  \dd {\mathcal{C}(\bm{\phi}^n)}{\bm{\phi}^n}_\parallel -\beta^n(\bm{\phi}^n-\bm{\phi}^{n-1})_\parallel\right),\\
 \Delta\bm{\phi^{n}}_\perp&=\lambda_\perp\left(\alpha^n  \dd {\mathcal{C}(\bm{\phi}^n)}{\bm{\phi}^n}_\perp-\beta^n(\bm{\phi}^n-\bm{\phi}^{n-1})_\perp\right),
 \end{align}
where $\lambda_\parallel$ and $\lambda_\perp$ are real constants between 0 and 1, we obtain
  \begin{equation}
    \begin{split}
- \lambda_\parallel\min\left(1,\frac{\gamma^n}{\alpha^n}\right)\dd {\mathcal{C}(\bm{\phi}^n)}{\bm{\phi}^n} \cdot\left(\alpha^n   \dd {\mathcal{C}(\bm{\phi}^n)}{\bm{\phi}^n}_\parallel -\beta^n(\bm{\phi}^n-\bm{\phi}^{n-1})_\parallel\right)\\ 
     -\frac{\gamma^n\lambda_\perp}{\alpha^n} \dd {\mathcal{C}(\bm{\phi}^n)}{\bm{\phi}^n} \cdot \left(\alpha^n  \dd {\mathcal{C}(\bm{\phi}^n)}{\bm{\phi}^n}_\perp-\beta^n(\bm{\phi}^n-\bm{\phi}^{n-1})_\perp \right) \\+\frac{L}{2} \left\Vert-\lambda_\parallel\min\left(1,\frac{\gamma^n}{\alpha^n}\right)\left(\alpha^n  \dd {\mathcal{C}(\bm{\phi}^n)}{\bm{\phi}^n}_\parallel -\beta^n(\bm{\phi}^n-\bm{\phi}^{n-1})_\parallel\right)\right\Vert^2 \\+ \frac{L}{2} \left\Vert-\frac{\gamma^n\lambda_\perp}{\alpha^n}\left(\alpha^n  \dd {\mathcal{C}(\bm{\phi}^n)}{\bm{\phi}^n}_\perp-\beta^n(\bm{\phi}^n-\bm{\phi}^{n-1})_\perp\right)\right\Vert^2 \leq 0.
    \end{split}
\end{equation}

Regrouping terms  :
\begin{equation}
    \begin{split}
\left( -\alpha^n \lambda_\parallel\min\left(1,\frac{\gamma^n}{\alpha^n}\right)  + 0.5 {\alpha^n}^2 L \lambda_\parallel^2 \min\left(1,\frac{\gamma^n}{\alpha^n}\right)^2\right) \left\Vert\dd {\mathcal{C}(\bm{\phi}^n)}{\bm{\phi}^n}_\parallel \right\Vert^2 \\+   \left( -\gamma^n \lambda_\perp  + 0.5 {\gamma^n}^2 L \lambda_\perp^2 \right)\left\Vert\dd {\mathcal{C}(\bm{\phi}^n)}{\bm{\phi}^n}_\perp \right\Vert^2  \\+\beta^n\left(  \lambda_\parallel\min\left(1,\frac{\gamma^n}{\alpha^n}\right)-\alpha^n L \lambda_\parallel^2\min\left(1,\frac{\gamma^n}{\alpha^n}\right)^2\right)\dd {\mathcal{C}(\bm{\phi}^n)}{\bm{\phi}^n}_\parallel \cdot (\bm{\phi}^n-\bm{\phi}^{n-1})_\parallel\\+\beta^n\left(\lambda_\perp \frac{\gamma^n}{\alpha^n}-\alpha^n  L \lambda_\perp^2\left(\frac{\gamma^n}{\alpha^n}\right)^2\right)\dd {\mathcal{C}(\bm{\phi}^n)}{\bm{\phi}^n}_\perp \cdot (\bm{\phi}^n-\bm{\phi}^{n-1})_\perp\\+  0.5 {\beta^n}^2 L \lambda_\parallel^2\min\left(1,\frac{\gamma^n}{\alpha^n}\right)^2\left\Vert (\bm{\phi}^n-\bm{\phi}^{n-1})_\parallel \right\Vert^2\\+0.5  {\beta^n}^2 L \lambda_\perp^2\left(\frac{\gamma^n}{\alpha^n}\right)^2 \left\Vert (\bm{\phi}^n-\bm{\phi}^{n-1})_\perp \right\Vert^2\leq 0.
\label{eq::developped_convergence_regroup_terms}
    \end{split}
\end{equation}

Three sufficient conditions are generated from Equation \eqref{eq::developped_convergence_regroup_terms} to guarantee the inequality. The first two ensure that the variation is negative, and the third one ensures that the variation of the solution due to the presence of inertia is smaller than the contribution of the gradient descent:
\begin{align}
    \lambda_\parallel \alpha^n\min\left(1,\frac{\gamma^n}{\alpha^n}\right)&\leq\frac{2}{L},
   \label{eq::constraint_on_alpha}\\
    \lambda_\perp\gamma^n&\leq\frac{2}{L},
    \label{eq::constraint_on_gamma}
\end{align}
\begin{equation}
    \begin{split}
    {\Bigg\vert} \beta^n\left(  \lambda_\parallel\min\left(1,\frac{\gamma^n}{\alpha^n}\right)-\alpha^n L \lambda_\parallel^2\min\left(1,\frac{\gamma^n}{\alpha^n}\right)^2\right)\dd {\mathcal{C}(\bm{\phi}^n)}{\bm{\phi}^n}_\parallel \cdot (\bm{\phi}^n-\bm{\phi}^{n-1})_\parallel\\+\;\beta^n\left(\lambda_\perp \frac{\gamma^n}{\alpha^n}-\alpha^n  L \lambda_\perp^2\left(\frac{\gamma^n}{\alpha^n}\right)^2\right)\dd {\mathcal{C}(\bm{\phi}^n)}{\bm{\phi}^n}_\perp \cdot (\bm{\phi}^n-\bm{\phi}^{n-1})_\perp\\+\;  0.5 {\beta^n}^2 L \lambda_\parallel^2\min\left(1,\frac{\gamma^n}{\alpha^n}\right)^2\left\Vert (\bm{\phi}^n-\bm{\phi}^{n-1})_\parallel \right\Vert^2\\+\;0.5  {\beta^n}^2 L \lambda_\perp^2\left(\frac{\gamma^n}{\alpha^n}\right)^2 \left\Vert (\bm{\phi}^n-\bm{\phi}^{n-1})_\perp \right\Vert^2{\Bigg\vert}\\ \leq{\Bigg\vert} \left( -\alpha^n \lambda_\parallel\min\left(1,\frac{\gamma^n}{\alpha^n}\right)  + 0.5 {\alpha^n}^2 L \lambda_\parallel^2 \min\left(1,\frac{\gamma^n}{\alpha^n}\right)^2\right) \left\Vert\dd {\mathcal{C}(\bm{\phi}^n)}{\bm{\phi}^n}_\parallel \right\Vert^2 \\+   \left( -\gamma^n \lambda_\perp  + 0.5 {\gamma^n}^2 L \lambda_\perp^2 \right)\left\Vert\dd {\mathcal{C}(\bm{\phi}^n)}{\bm{\phi}^n}_\perp \right\Vert^2 {\Bigg\vert}.
    \label{eq::constraint_inertia}
    \end{split}
\end{equation}
Taking the coefficient $\lambda_\parallel=1$, $\lambda_\perp=1$, and the terms $\min\left(1,\frac{\gamma^n}{\alpha^n}\right)=1$ as they are the most constraining values for the size of the step $\alpha$ and $\gamma$, the Inequations \eqref{eq::constraint_on_alpha} and \eqref{eq::constraint_on_gamma} simplify to a single inequation: 
\begin{align}
    \gamma^n&\leq\frac{2}{L}.
    \label{eq::constraint_gamma_final}
\end{align}

Inequation \eqref{eq::constraint_inertia} can be hard to evaluate; as such, a good estimate of the bound on $\beta$ can be obtained by simplifying this inequation using the assumption that the active set of constraints is empty, meaning $\lambda_\parallel=1$ and  $\lambda_\perp=1$.Additionally, If the same kind of formula is used to define $\alpha^n$ and $\gamma^n$, it also implies that  $\alpha^n \approx \gamma^n$ as the gradient of the cost function dominates the Lagrangian gradient. Applying these assumptions reduces the Inequation \eqref{eq::constraint_inertia} to:
\begin{align}
    \beta^n \leq &\; \frac{1}{L \left\Vert \bm{\phi}^n - \bm{\phi}^{n-1} \right\Vert^2} \left(
       -(1-\alpha^n L) \dd {\mathcal{C}(\bm{\phi}^n)}{\bm{\phi}^n} \cdot (\bm{\phi}^n - \bm{\phi}^{n-1}) \right. \nonumber \\
    & \left. + \left( 
            \left( (1-\alpha^n L) \dd {\mathcal{C}(\bm{\phi}^n)}{\bm{\phi}^n} 
            \cdot (\bm{\phi}^n - \bm{\phi}^{n-1}) \right)^2 \right. \right. \nonumber \\
    & \left. \left. +\; 2 L \left\Vert \bm{\phi}^n - \bm{\phi}^{n-1} \right\Vert^2 
            \left| -\alpha^n + 0.5 (\alpha^n)^2 L \right| 
            \left\Vert \dd {\mathcal{C}(\bm{\phi}^n)}{\bm{\phi}^n} \right\Vert^2 
        \right)^{\frac{1}{2}} \right).
    \label{eq::constraint_first_simplification}
\end{align}

 Additionally, if it is assumed that $\dd {\mathcal{C}(\bm{\phi}^n)}{\bm{\phi}^n} \cdot (\bm{\phi}^n-\bm{\phi}^{n-1}) \leq 0 $ or that $\alpha^n\approx \frac{1}{L}$
 , the inequation can be further simplified to:
\begin{align}
    \begin{split}
{\beta^n}\leq \sqrt{\left\vert {\alpha^n}^2-\frac{2\alpha^n}{L}\right\vert} \frac{ \left\Vert\dd {\mathcal{C}(\bm{\phi}^n)}{\bm{\phi}^n}\right\Vert }{\left\Vert (\bm{\phi}^n-\bm{\phi}^{n-1}) \right\Vert}.
    \label{eq::constraint_first_simplification }
    \end{split}
\end{align}
These bounds only guarantee convergence when using $\alpha^n=\gamma^n$. Note here that $\gamma^n$ is defined using the variation of an estimation of the Lagrangian gradient, which is evaluated after the projection. In practice, $\gamma^n \neq \alpha^n$ in most cases. However, they are generally relatively close since the gradient of the Lagrangian is generally mostly controlled by the gradient of the cost function. Using these bounds along with a good estimate of $L$ enables the definition of the different step sizes ($\alpha^n$, $\gamma^n$, $\beta^n$). A generally good local estimate of $L$ is\cite{boyd2004convex}: 
\begin{align}
    L \approx \frac{\Vert\dd {\mathcal{C}(\bm{\phi}^n)}{\bm{\phi}^n} - \dd {\mathcal{C}(\bm{\phi}^{n-1})}{\bm{\phi}^{n-1}}\Vert}{\Vert(\bm{\phi}^n - \bm{\phi}^{n-1})\Vert}.
\end{align}
From this estimate, and by maximizing the gradient descent portion of the cost function, $\alpha^n$ is defined as: 
\begin{align}
    \alpha^n = \frac{1}{L}=\frac{\Vert(\bm{\phi}^n - \bm{\phi}^{n-1})\Vert}{\Vert\dd {\mathcal{C}(\bm{\phi}^n)}{\bm{\phi}^n} - \dd {\mathcal{C}(\bm{\phi}^{n-1})}{\bm{\phi}^{n-1}}\Vert}.
    \label{eq::alpha_def}
\end{align}
From this definition, to be consistent with our hypothesis that $\gamma^n=\alpha^n$ in the absence of constraint, $\gamma^n$ is defined as:
\begin{align}
    \gamma^n = \frac{\Vert(\bm{\phi}^n - \bm{\phi}^{n-1})\Vert}{\Vert\dd {\mathcal{L}(\bm{\phi}^n)}{\bm{\phi}^n} - \dd {\mathcal{L}(\bm{\phi}^{n-1})}{\bm{\phi}^{n-1}}\Vert}.
\end{align}
Finally, using the definition of $\alpha^n$, $\beta^n$ is defined as:
\begin{align}
{\beta^n}=\hat{\beta}\alpha^n \frac{ \left\Vert\dd {\mathcal{C}(\bm{\phi}^n)}{\bm{\phi}^n}\right\Vert }{\left\Vert (\bm{\phi}^n-\bm{\phi}^{n-1}) \right\Vert},
\end{align}
where $\hat{\beta}$ is a user-defined parameter between 0 and 1. A $\hat{\beta}$ value around 0.2 gives adequate results.

Note that in practice, the choice of Lipchsitz constant does not guarantee convergence, as it is only a local approximation that can carry a non-negligible error. Some alternatives, such as the approach proposed by Nishioka \textit{et al.} \cite{nishioka2023inertial}, can be used to guarantee convergence by increasing the estimation of the Lipchsitz constant when an increase in the cost function is observed. As such, a simple strategy would be to repeat the iteration if it leads to an increase in the cost function using a multiple of the current Lipchsitz constant estimate, effectively reducing the step size. However, in practice, this was not required for the case presented in the present document. 
Additionally, in the absence of previous step information to evaluate the Lipchsitz constant, a small conservative step is chosen based on the univariate bounds ($d \leq \phi_i \leq e$) and the $L_\infty$ norm of the gradient:
\begin{align}
    \alpha^n = \frac{0.1 (e-d)}{\left\Vert \dd {\mathcal{C}(\bm{\phi}^n)}{\bm{\phi}^n}\right\Vert_\infty}.
    \label{eq::alpha_def_new}
\end{align}
In such cases, $ \gamma^n= \alpha^n$ and $\beta^n=0$.

\subsubsection{Implementation}
An implementation of the algorithm described in the previous sections can be found in the supplementary material of the presented document. This implementation is realized in Python and relies on the NumPy library \cite{harris2020array}. This implementation aims to give the reader a better understanding of the algorithm part of the proposed method. A small implementation for a general convex minimization problem with randomly defined linear equality and inequality constraints accompanies it. 
\bmsection{Examples of Applications}
\label{sec::application}

Two 3D heat conduction topology optimization problems based on heatsink design are utilized to illustrate the capabilities of the PGD algorithm. This section outlines the specific optimization problem and details the solver employed for the topology optimization. Additionally, comparisons with other algorithms are included, demonstrating the performance differences for the presented test case.

\subsection{Topology optimization framework}

As the introduction mentions, the topology optimization algorithm is implemented in DFEMwork \cite{audet2008dfemwork}, a proprietary finite element solver of the National Research Council of Canada. The optimizer described here is implemented in the same numerical architecture, avoiding costly communication and file-writing processes. For more detail on the solver and the implementation of the adjoint problem solved to obtain the gradient of the cost function, we refer the reader to \cite{navahThermofluidTopologyOptimization2024,lamarche2024additively}.

\subsection{Case description}
The problem under consideration is the stationary heat equation with variable thermal conductivity $\kappa$, viz.
\begin{align}
    \nabla \cdot \left(\kappa \nabla T\right) = S,
    \label{eq::heat_eq_pde}
\end{align}
which is solved for the temperature field $T$. The optimization problem involves determining the optimal distribution of conductive and insulation materials of respective thermal conductivity $\kcond=1$ and $\kins=0.001$, which minimizes the average temperature in the domain $\Omega$. The problem is relaxed through the use of a continuous pseudo density field, $\Bar{\phi}$, which varies from 0 to 1 and corresponds to the volume fraction of the conductive material. This translates into the following cost function: 
\begin{align}
    \mathcal{C}(\Bar{\phi}, T(\Bar{\phi})) = \frac{\int_\Omega T(\Bar{\phi}) d\Omega}{\int_\Omega  d\Omega}.
    \label{eq::cost_function}
\end{align}

The volume fraction results from a three-field projection, $\phi \to \hat{\phi} \to \Bar{\phi}$, where the design variable $\phi$ is first filtered through the solving of a Helmholtz-type equation \cite{lazarovFiltersTopologyOptimization2011} and then projected using a smoothed Heaviside function \cite{wang2011projection}:
\begin{align}
    H(x)=\frac{\tanh(\lambda (x-0.5))+\tanh(\lambda/2) }{2\tanh(\lambda/2)}.
    \label{eq::heaviside}
\end{align}
The resulting volume fraction is then used to determine the material properties using a SIMP \cite{bendsoe1988generating} interpolation scheme, $\kappa = \kins + (\kcond - \kins)\Bar{\phi}^b$, where $b$ is a penalty parameter. This combination of filter and projections penalizes intermediate values of the volume fraction while preventing checkerboard-like material distributions. Here, the filter radius used in the Helmholtz equation is three times the characteristic length of the elements. The sensitivity of the cost function relative to the design variable is evaluated using the adjoint method as described in \cite{lamarche-gagnon2021comparative}. 

The minimization problem is subject to a constraint on the maximal volume of conductive material in the domain, $a_0$:
\begin{align}
    f_0 - a_0  = \frac{\int_\Omega \bar{\phi} d\Omega}{\int_\Omega  d\Omega}-a_0\leq 0.
    \label{eq::material_constraint}
\end{align}

To test the capacity of the optimizer to handle nonlinear constraints, we also consider an overhang-limitation constraint \cite{qian2017undercut}. This constraint was developed to improve the manufacturability of parts intended to be fabricated by additive manufacturing technologies. The constraint limits the angle of features relative to the given build direction, resulting in self-supporting designs. In this work, the constraint is evaluated using the intermediate variable $\hat{\phi}$ (as in \cite{lamarche2024additively}) and is given by:
\begin{align}
     f_1 - a_1 = \frac{\int_\Omega H'\left(\frac{\nabla \hat{\phi} \cdot \bm{n}}{\left\Vert \nabla \hat{\phi} \right\Vert }- \cos(\theta_0)\right )\nabla \hat{\phi} \cdot \bm{n} d\Omega}{\int_\Omega  d\Omega}-a_1\leq 0,
    \label{eq::overhang_constraint}
\end{align}
where $H'(x)=\frac{1}{1+e^{-20 x}}$ is a smoothed Heaviside function, $\theta_0$ is the maximal angle relative to the build direction, and $a_1$ is a user-defined limit on the overhang indicator. In all cases, $a_1$ is set to approximately 10\% of the overhang indicator value of the optimal solution obtained without this constraint, and the angle is set to $\theta_0=\frac{\pi}{4}$.

The heat conduction problem is solved in a rectangular cuboid domain, shown schematically in Figure \ref{fig::heatsink_setup}, with a uniform volumetric heat source ($S=1$). The temperature of a small portion of the bottom surface is set to $T=0$, while the remaining surfaces are adiabatic. Equations \eqref{eq::heat_eq_pde} to \eqref{eq::overhang_constraint} are discretized using the finite element method and the computational mesh is generated using the  Gmsh software \cite{gmsh}, with a characteristic length of 0.02 for a total of 26,601 degrees of freedom (143,115 elements). 

\begin{figure}[!htpb]
	\centering
	\includegraphics[width=0.3\textwidth]{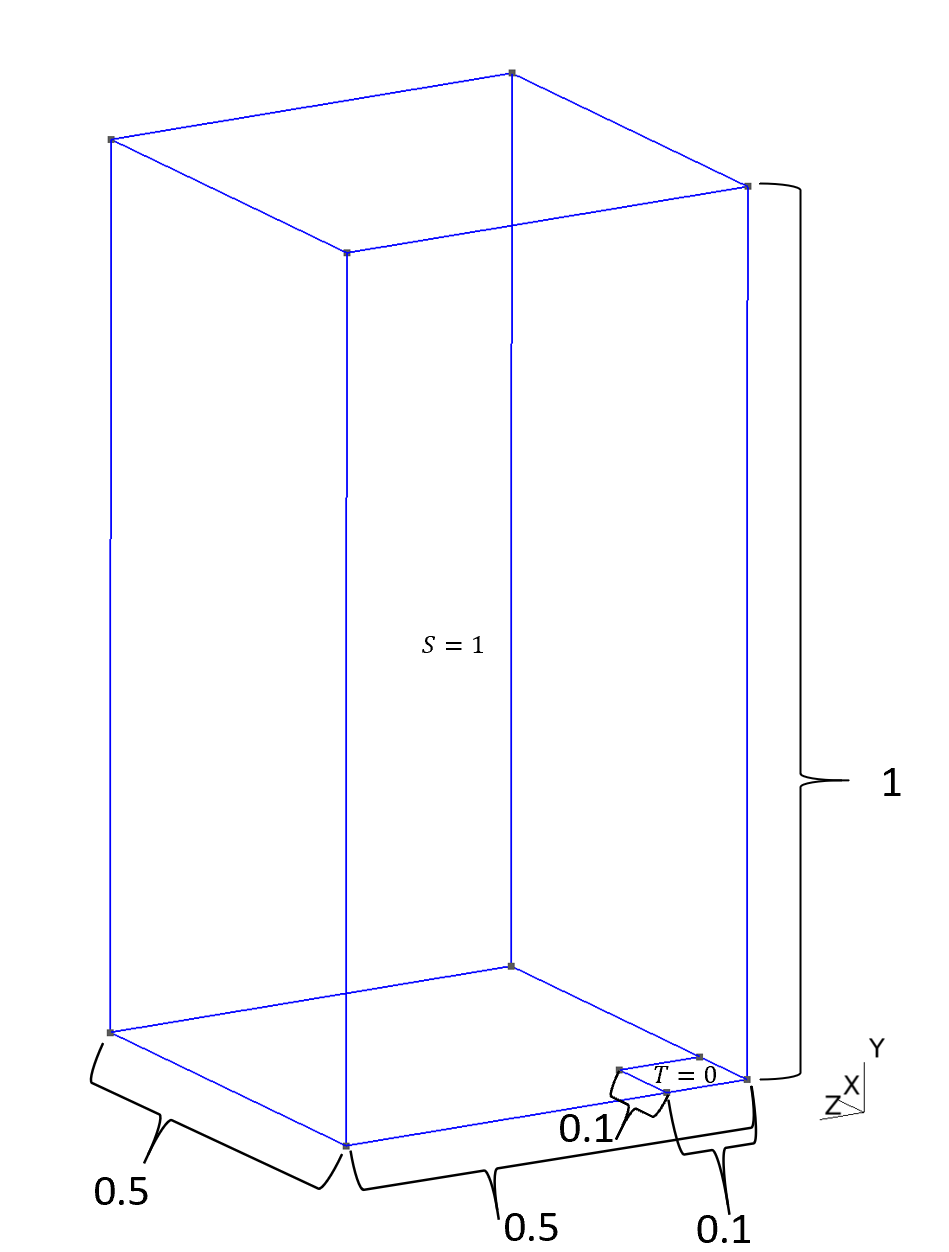}
	\caption{Representation of the computational domain of the 3D heatsink case. }
\label{fig::heatsink_setup}

\end{figure}

For all cases, except if mentioned otherwise, the tolerance on the constraints used ($\epsilon_i$, cf. Section \ref{sec::consideration_kkt}) is set to 2\,\% of the constraint constant, which here corresponds to $\epsilon_0=0.002$ for the volume fraction constraint \eqref{eq::material_constraint}, and $\epsilon_1=0.000722$ for the overhang constraint \eqref{eq::overhang_constraint}.

\subsection{Optimization results}
We separate the optimization cases in two, \ie with and without the overhang constraint. These cases demonstrate the algorithm's capacity to optimize even in the presence of nonlinear constraints. The optimization problem is solved using a continuation approach for the SIMP penalty and Heaviside projection parameters, $b$ and $\lambda$, respectively. A maximum of 50 optimization cycles are performed for each pair of values of these parameters. If the relative variation of the cost function between two cycles is smaller than $10^{-6}$, then the solver goes to the next pair of $(b, \lambda)$. The hyper-parameters values are given in Table \ref{table::topoopt_beta_b}.
\begin{table}
\begin{center}
 \begin{tabular}{|c| c| c|c| c| c|c|c|c|} 
 \hline
   Optimization loop & 1 &2 & 3& 4& 5 & 6 & 7 & 8 \\ 
 \hline
 $b$ &	1 & 2 &	3	& 3& 3 & 3&3 &3  \\
 \hline 
 $\lambda$&	1 & 2 &	4	& 8 & 16  & 32  & 64  & 128 \\ 
 \hline
\end{tabular}
\end{center}
\caption{SIMP and Heaviside projection penalty parameters at a given optimization loop.}
\label{table::topoopt_beta_b}
\end{table}

\subsubsection{Results without the overhang constraint}
For comparison, the PGD algorithm is compared to the results obtained using i) the Method of Moving Asymptote (MMA), as described by Svanberg \cite{svanberg1987method} and using an asymptote move limit of 0.1\footnote{This value was found to provide the best convergence properties for the problem being addressed. Larger values lead to oscillation, and smaller values lead to a slower convergence.} and ii) the GCMMA algorithm implemented in the NLopt library \cite{NLopt} (using the solver named LD\_MMA). Additionally, the current PGD algorithm with its default parameters $( \mu=0.95, \: \hat{\beta}=0.2)$ is compared to a ``traditional'' PGD implementation, \ie without: 1) the inertia consideration ($\hat{\beta}=0$), 2) the relaxation for broken constraint ($\mu=1$), and 3) the step size modification based on the variation of the Lagrangian ($\frac{\gamma}{\alpha}=1$). The results are also compared with a version of the algorithm with inertia and relaxation for broken constraints, but without the modification of the step based on the variation of the Lagrangian $(\frac{\gamma}{\alpha}=1, \:  \mu=0.95, \: \hat{\beta}=0.2)$; this version is referred to as the ``intermediary'' implementation of PGD.

The results, shown in Figures \ref{fig::cost_no_overhang} and \ref{fig::sf_no_overhang}, show that, for this case, the current PGD algorithm performs similarly to all the other methods tested. Small improvements relative to the MMA algorithm can be observed, both in terms of the final cost function value and convergence speed. Regarding the different PGD implementations, the use of the $\gamma$ factor did not significantly improve the solution compared to the intermediary version. However, the traditional implementation converges to the worst local minima (2\,\% higher than MMA). This highlights the effect of the inertial step, enabling the solver to avoid converging to local minima to a certain extent. The NLopt results present a bit more oscillation and slower convergence than the other approaches in this case, along with a 4\,\% increase of the final cost function value relative to the MMA approach.

For all algorithms, the volume fraction inequality constraint is satisfied without significant overshoots, except when the optimization parameters are updated, which is explained by the non-linearity brought by the smoothed Heaviside projection (Equation \eqref{eq::heaviside}). The final design obtained for the MMA and PGD (default) algorithms is presented in Figure \ref{fig::design_comparaison}. The figure shows noticeable differences between the two designs, indicating that the algorithms converged to two local minima. However, the general branching pattern and the size of the branches remain similar in both cases.

\begin{figure}[!htpb]
	\centering
	\includegraphics[width=1\textwidth]{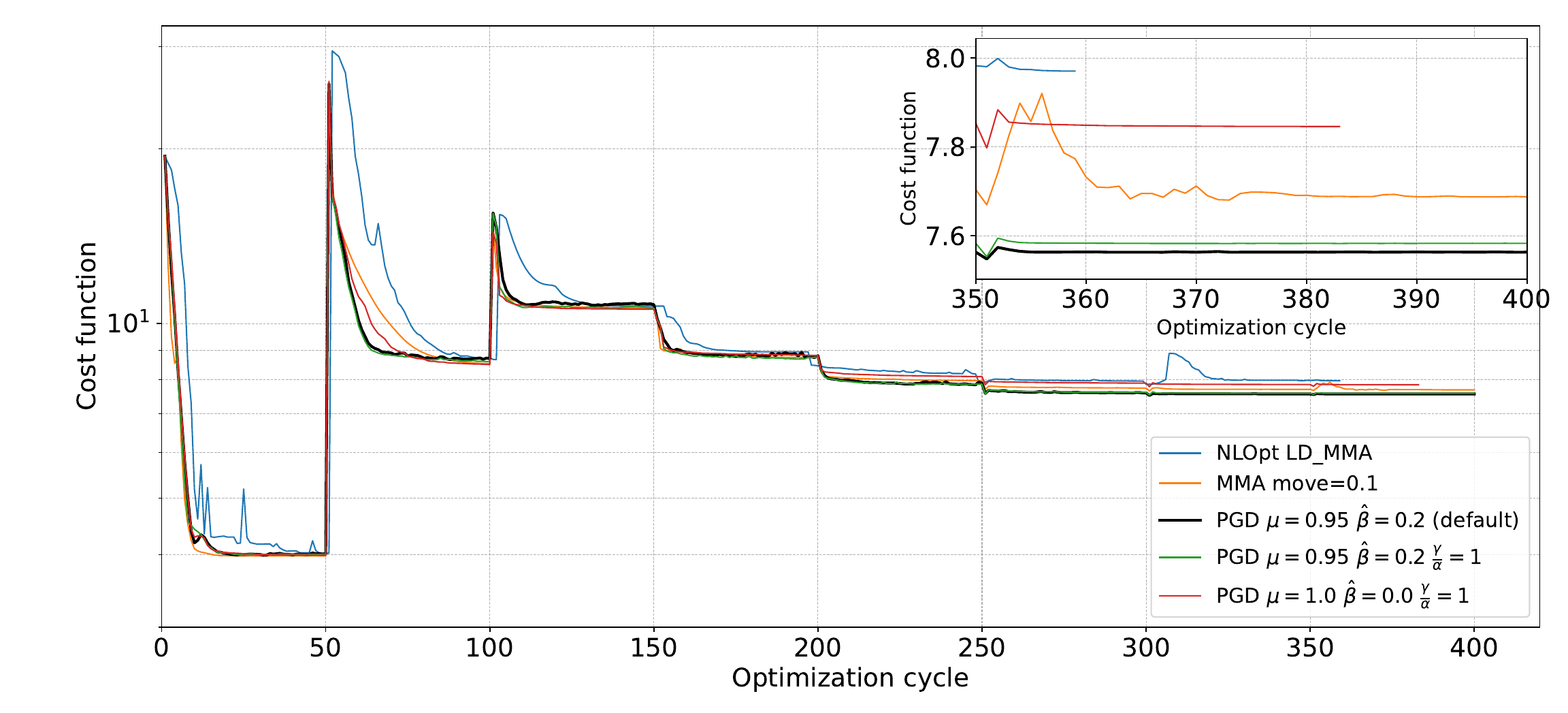}
	\caption{Evolutions of the cost function for the different optimization algorithms (without the overhang constraint).}
\label{fig::cost_no_overhang}
\end{figure}

\begin{figure}[!htpb]
	\centering
	\includegraphics[width=0.933333\textwidth]{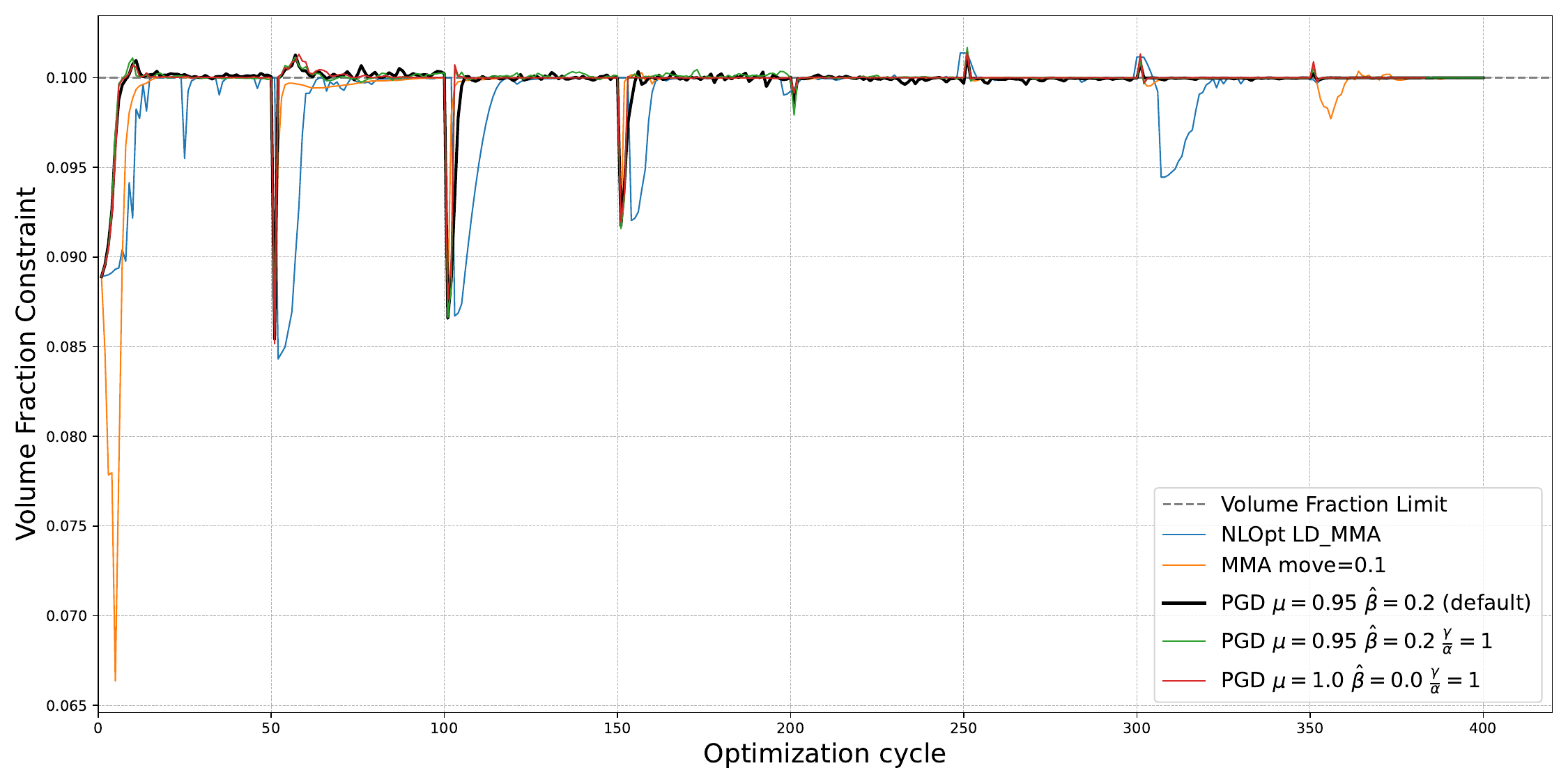}
	\caption{Evolutions of the volume fraction constraint for the different optimization algorithms (without the overhang constraint).}
\label{fig::sf_no_overhang}
\end{figure}

\begin{figure}[!htpb]
	\centering
    \begin{subfigure}{0.5\textwidth}
      \centering
    	\includegraphics[width=1\textwidth]{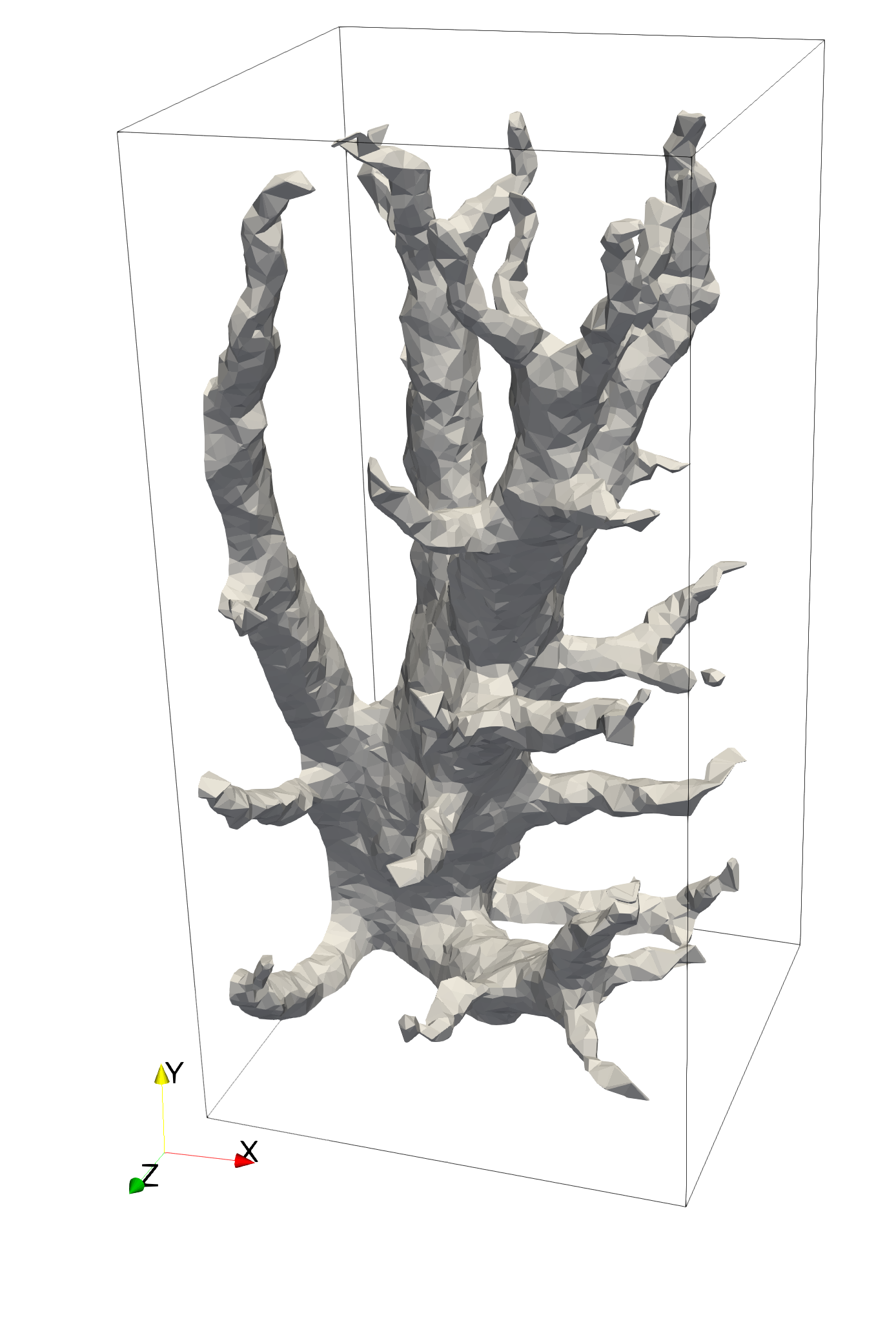}
     \caption{} 
    \end{subfigure}%
    \begin{subfigure}{0.5\textwidth}
      \centering
    	\includegraphics[width=1\textwidth]{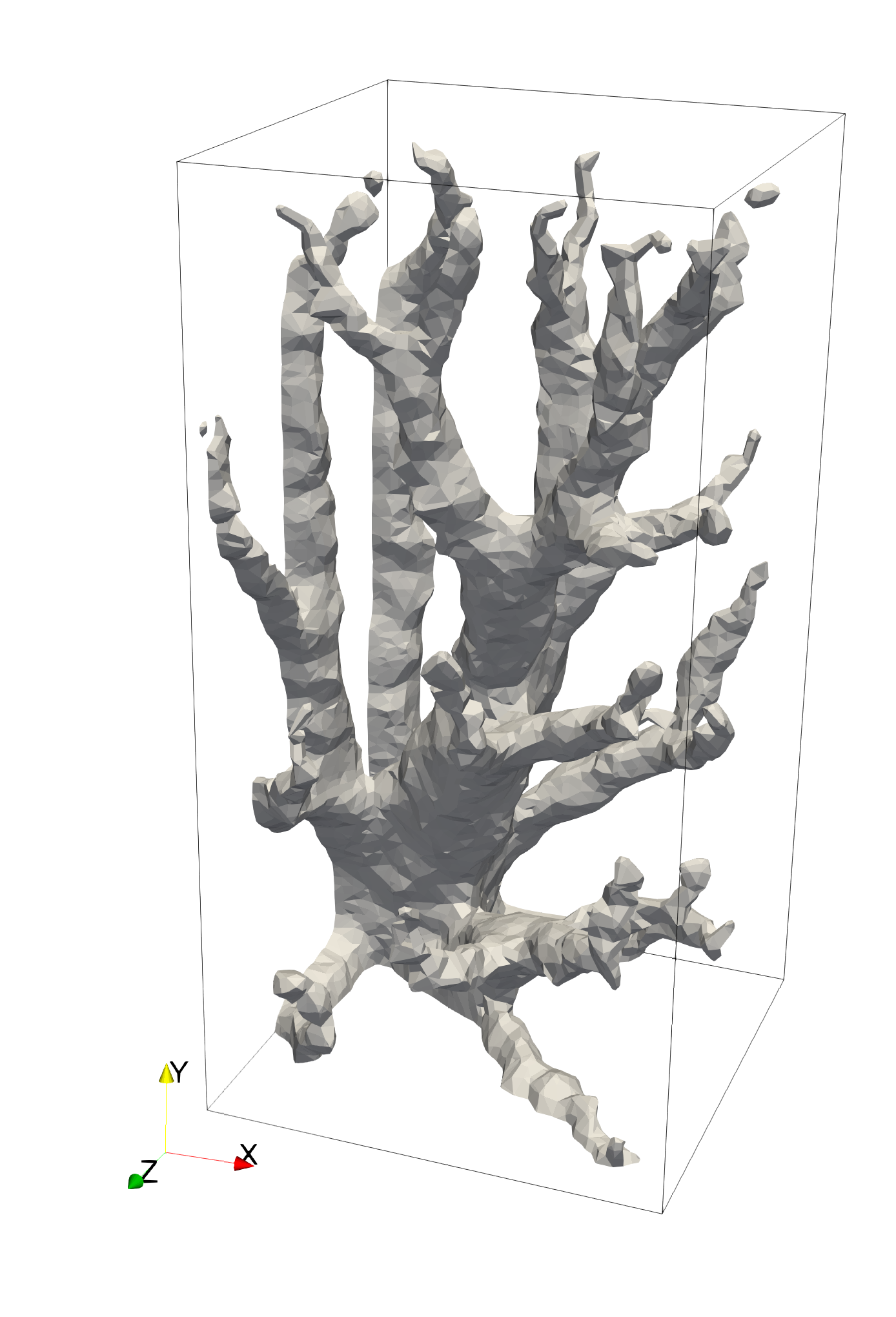}
     \caption{} 
    \end{subfigure}
\caption{Comparison of the final design obtained using the (a) MMA and (b) PGD (proposed with default parameters) algorithms, as depicted by the isosurface $\hat{\phi} = 0.5$ (without the overhang constraint). \label{fig::design_comparaison}}
\end{figure}

\subsubsection{Results with the overhang constraint}
The overhang constraint (Equation \eqref{eq::overhang_constraint}) is added for this case. This constraint is highly nonlinear and, as such, can be hard to maintain for a small value of $a_1$. The same three variants of the PGD algorithm are again compared to the MMA method, although this time using a move limit of 0.05 (a larger move resulted in a larger final cost function and a slower convergence). NLopt optimization results are not shown for this case due to their highly oscillating nature.  The results are presented in Figures \ref{fig::overhang_cost}, \ref{fig::overhang_sf} and \ref{fig::overhang_overhang}. Figure \ref{fig::overhang_cost} first shows that all three PGD algorithms converge substantially faster than MMA in the second loop (cycles 51 to 100). The faster convergence is also associated with a breaking of the overhang constraint (Figure \ref{fig::overhang_overhang}) at the beginning of the loop, which is then gradually recovered. Note that the inclusion of the $\gamma$ term accelerates the recovery of the constraint after it is broken. In the third loop (cycles 101 to 150), the traditional PGD algorithm exhibits the fastest convergence and achieves the lowest cost function value. However, the faster convergence is accompanied by a considerable overshoot of the overhang constraint, which is violated for the whole third loop. This demonstrates that adding the relaxation parameter, $\mu$, for constraint breaking enables the PGD algorithm to converge without significantly violating the constraints. The MMA algorithm also presents a significant overhang constraint breaking for the first half of the third loop, which is then quickly recovered. In the remaining loops, the proposed PGD algorithm with default parameters converges to slightly lower cost function values than all other methods. Adding $\gamma$ helps slightly to converge faster to a lower cost in the presence of highly nonlinear constraints but also leads to some oscillation under the volume fraction limit due to a slightly more aggressive step. These results highlight how including $\gamma$ helps consider the effect of non-linear constraint in the step size and accelerate the convergence while reducing constraint breakage. 
Similar to the case without the overhang constraint, the volume constraint is well respected for all of the algorithms, with the exception of slight overshoots for the PDG algorithms in loops 1 and 2. The final design obtained for the MMA and the proposed PGD algorithm using the default parameter is presented in Figure \ref{fig::design_comparaison_with_overhang}. As for the previous case, the designs produced by both algorithms are different, suggesting that each algorithm has converged to a distinct local minimum but of comparable cooling efficiency.

\begin{figure}[!htpb]
	\centering
	\includegraphics[width=1\textwidth]{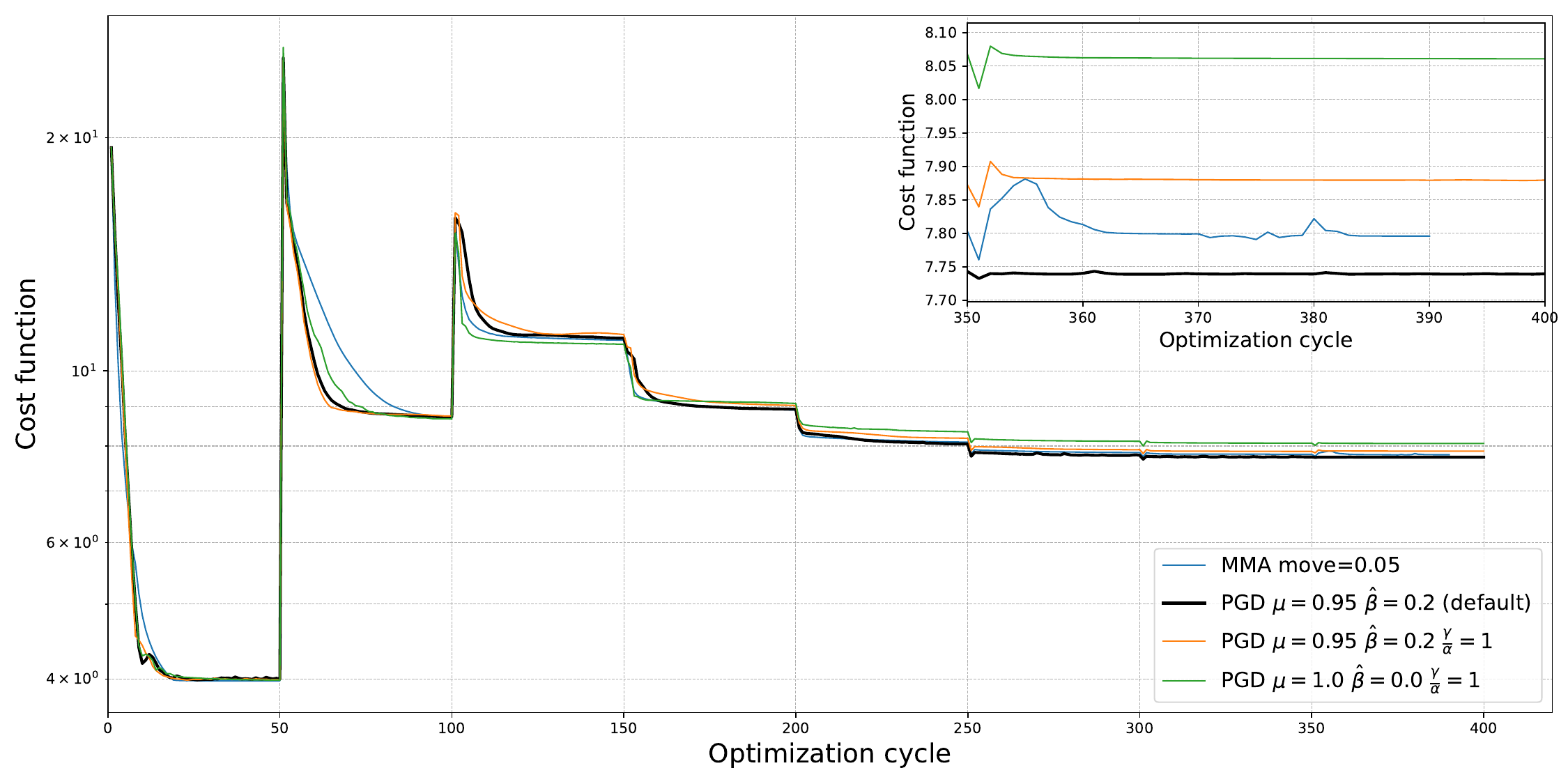}
	\caption{Comparison of the cost function evolution for different optimization algorithms (with the overhang constraint).}
\label{fig::overhang_cost}
\end{figure}

\begin{figure}[!htpb]
	\centering
	\includegraphics[width=1\textwidth]{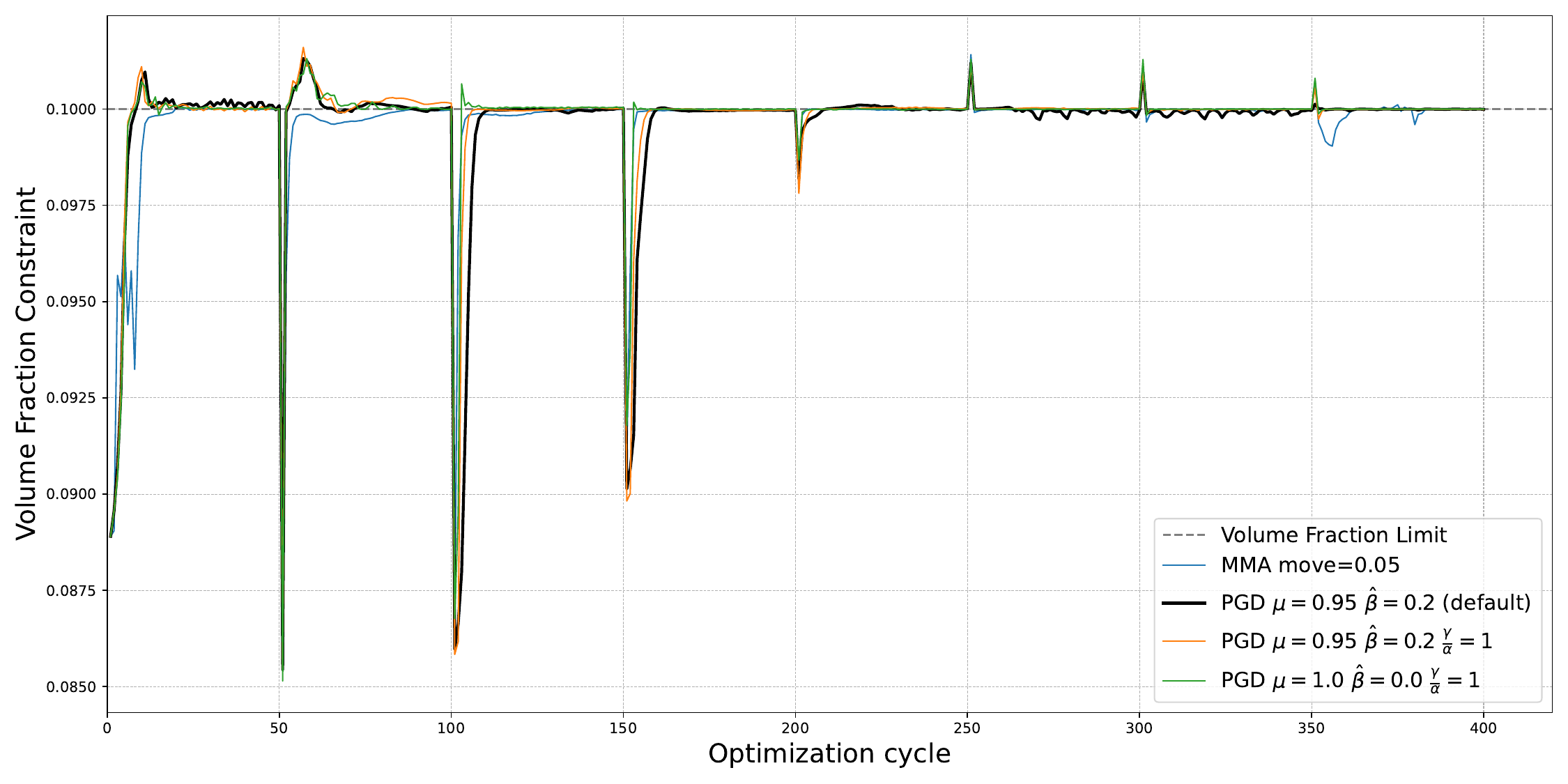}
	\caption{Comparison of the volume fraction constraint evolution for different optimization algorithms (with the overhang constraint).}
\label{fig::overhang_sf}
\end{figure}
\begin{figure}[!htpb]
	\centering
	\includegraphics[width=1\textwidth]{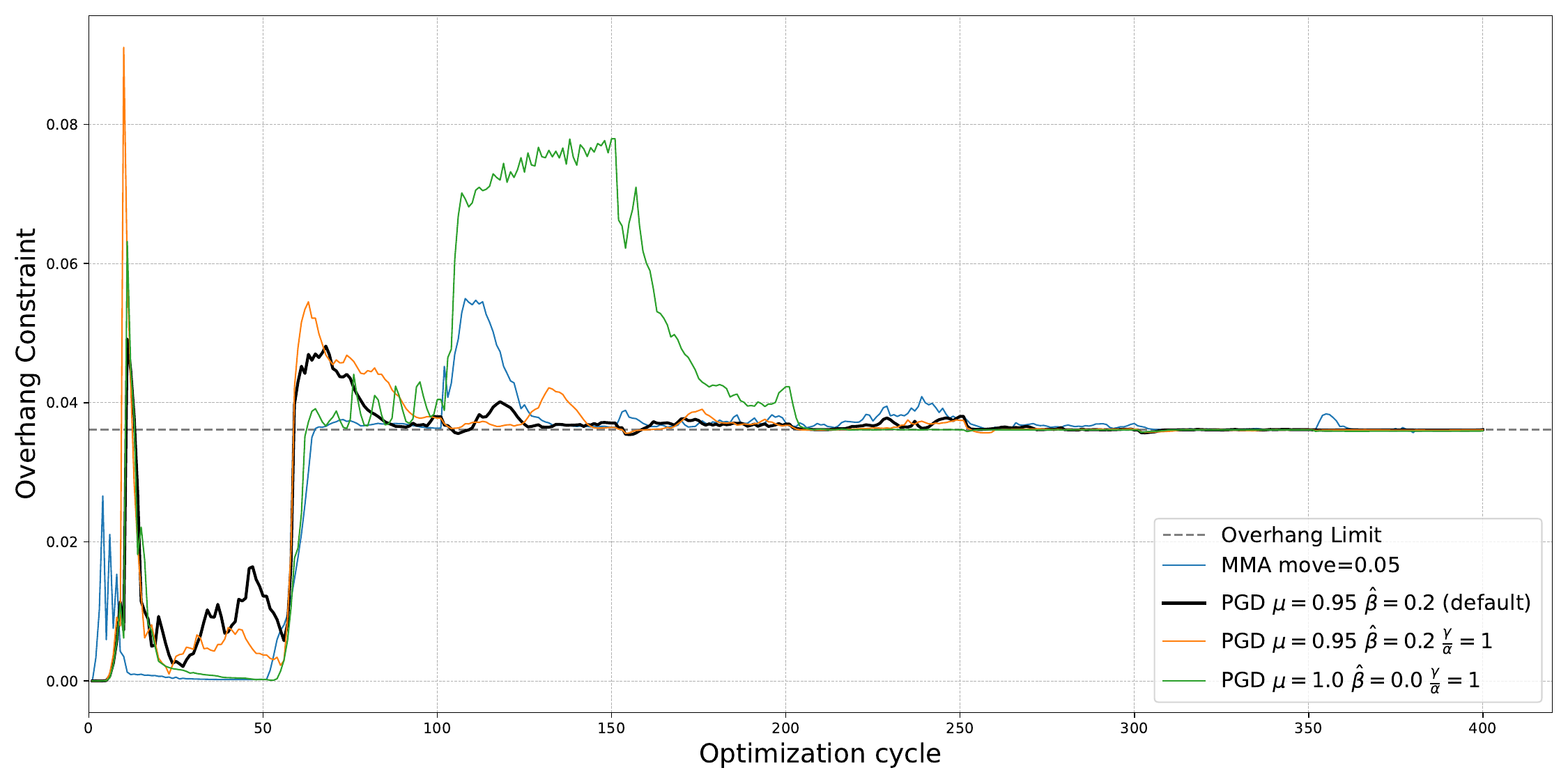}
	\caption{Comparison of the overhang constraint evolution for the different optimization algorithms.}
\label{fig::overhang_overhang}
\end{figure}

\begin{figure}[!htpb]
	\centering
    \begin{subfigure}{0.5\textwidth}
      \centering
    	\includegraphics[width=1\textwidth]{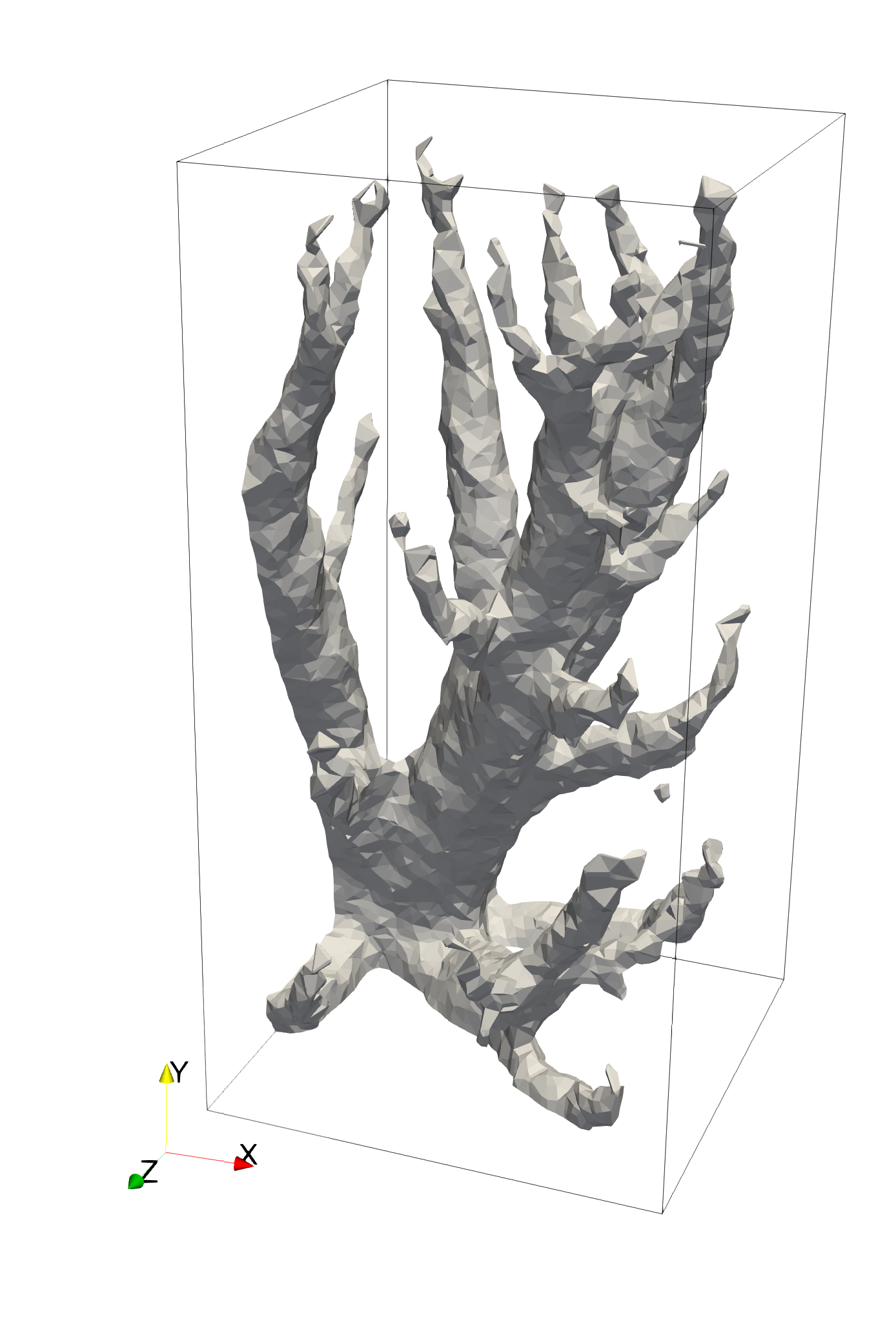}
     \caption{} 
    \end{subfigure}%
    \begin{subfigure}{0.5\textwidth}
      \centering
    	\includegraphics[width=1\textwidth]{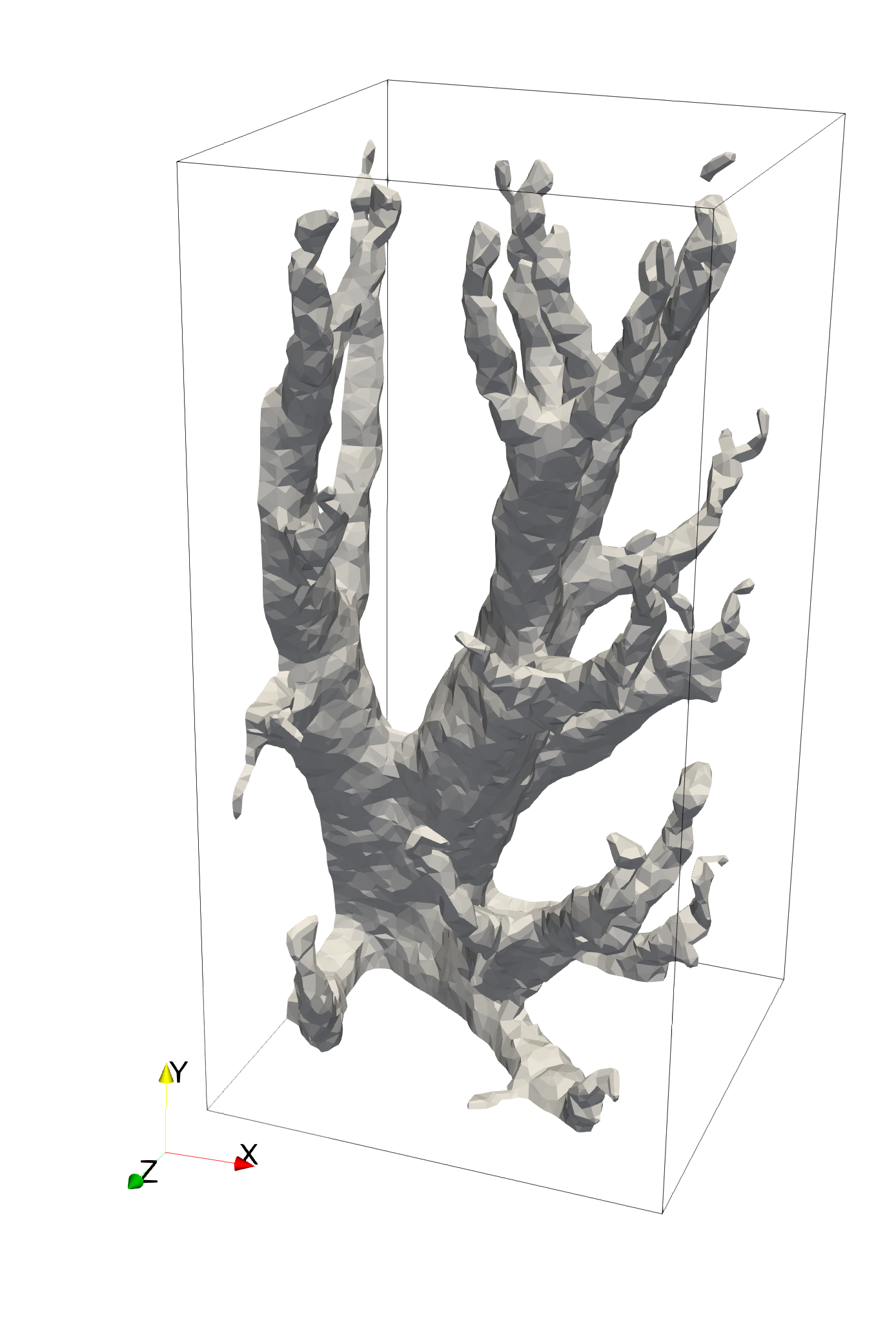}
     \caption{} 
    \end{subfigure}
\caption{Comparison of the final design obtained using the (a) MMA and (b) PGD (proposed, with default parameter) algorithms, as depicted by the isosurface $\hat{\phi} = 0.5$ (with the overhang constraint). \label{fig::design_comparaison_with_overhang}} 
\end{figure}

\subsection{Effect of the PGD parameters}
The proposed PGD algorithm has a few parameters that are available to the user. The main parameters the user has to define are the inertia parameter $\hat{\beta}$, the broken constraint relaxation $\mu$, and the tolerance $\epsilon_i$ on each constraint. This section will highlight the effect of these parameters on the optimizer's convergence using the case with the overhang constraint.

\subsubsection{Effect of the constraints tolerance, $\epsilon_i$}
The tolerance on each constraint determines when a constraint is considered broken and, as such, when the exponent on $\mu$ must be increased or decreased. We recall that the proposed tolerance was $\epsilon_i = 0.02 a_i$ (\ie $a_0=0.1$ and $a_1=0.0361$). Optimization results  using $\frac{\epsilon_i}{a_i}=0.01$, $\frac{\epsilon_i}{a_i}=0.001$, $\frac{\epsilon_i}{a_i}=0.0001$ are presented in Figures \ref{fig::overhang_cost_tol_effect} and \ref{fig::overhang_overhang_tol_effect}. Note that, in this and subsequent sections, as the volume fraction constraint is well respected for all cases, the evolution of this constraint is not presented. The results clearly show that as the tolerance of the constraints is tightened, the cycles are significantly relaxed, leading to a slower optimization and slightly larger cost function values. However, as the tolerance is lowered, the overshoots on the constraints are significantly reduced. Only changes in optimization parameters (see Table \ref{table::topoopt_beta_b}) cause the constraints to go well above their respective tolerances. It is also possible to observe that the optimization process keeps the constraint value closer to the limit as the tolerance is reduced. This means that the tolerances enable the user to control the acceptable constraint error at the cost of the convergence speed of the algorithm. In the presence of highly nonlinear constraints such as the overhang constraint, it is suggested to use a larger tolerance if possible to favor the minimization process and avoid over-relaxing.

\begin{figure}[!htpb]
	\centering
	\includegraphics[width=1\textwidth]{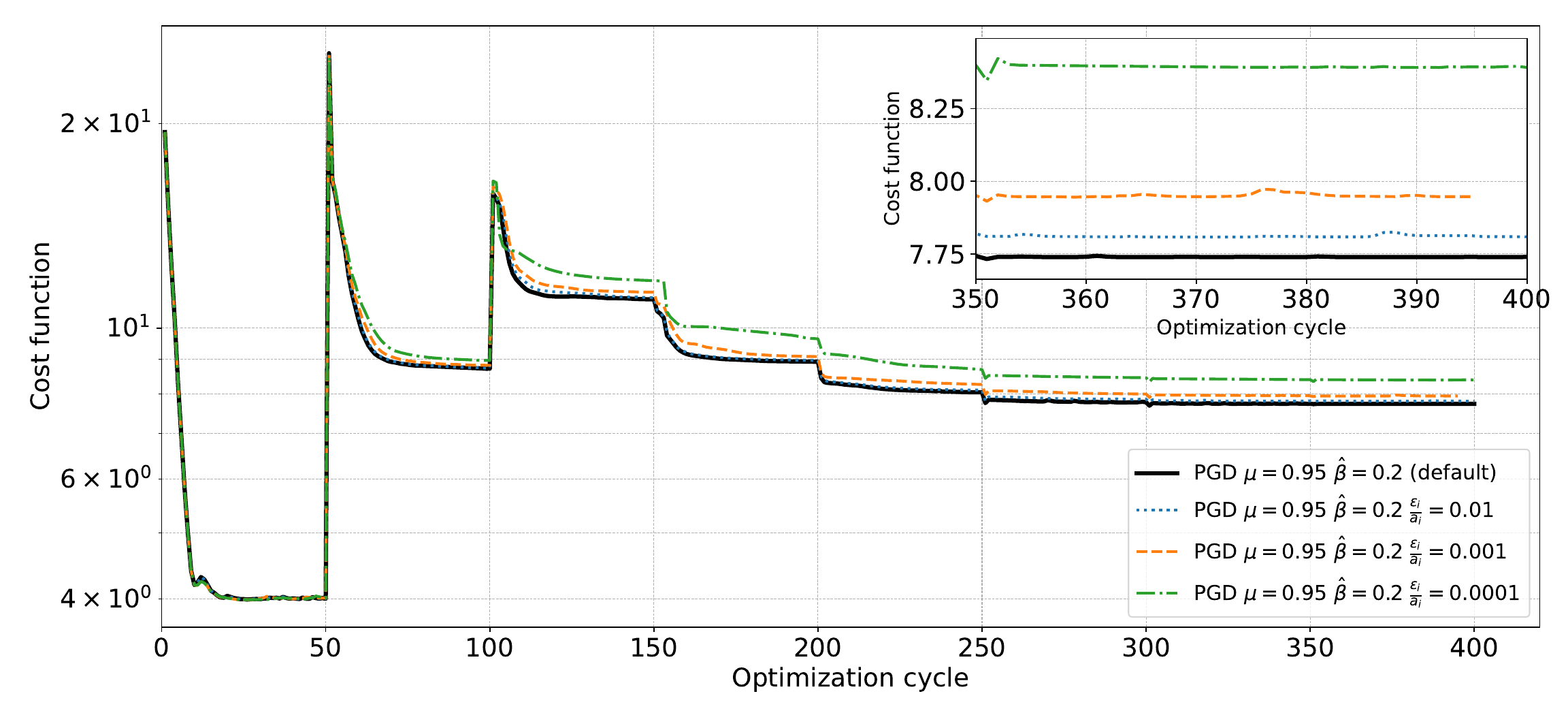}
	\caption{Comparison of the cost function evolution for different values of the relative tolerance $\frac{\epsilon_i}{a_i}$ on the constraints.}
\label{fig::overhang_cost_tol_effect}
\end{figure}

\begin{figure}[!htpb]
	\centering
	\includegraphics[width=1\textwidth]{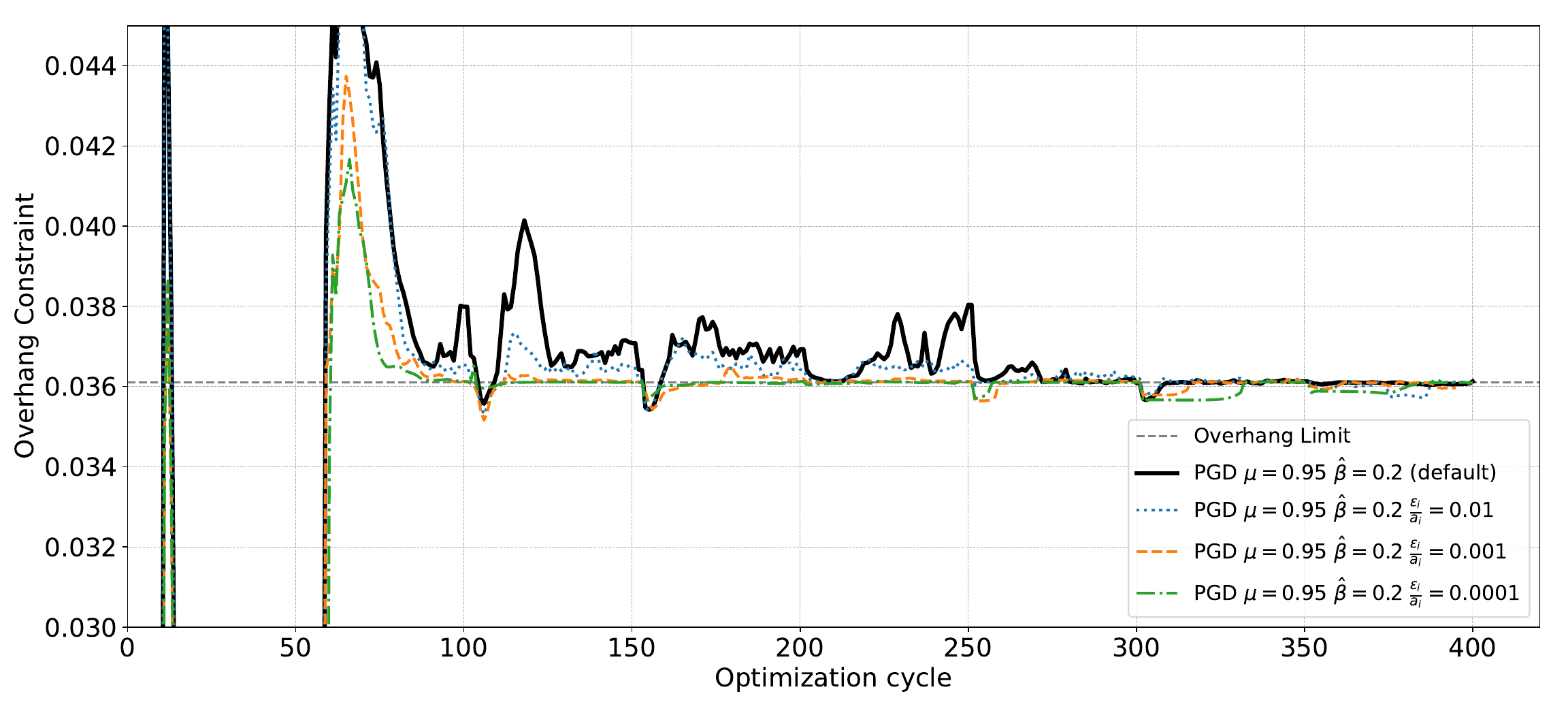}
	\caption{Comparison of the overhang constraint for different values of the relative tolerance $\frac{\epsilon_i}{a_i}$ on the constraint.}
\label{fig::overhang_overhang_tol_effect}
\end{figure}

\subsubsection{Effect the broken constraint relaxation factor, $\mu$}
The broken constraint relaxation factor (cf. Equation \eqref{eq::update_with_mu}) determines how fast the relaxation will react to a broken constraint. The lower $\mu$, the more the optimization process will try to aggressively correct the constraint without minimizing the cost function when any constraint is defined as broken (value above the tolerance). To show how this parameter affects the evolution of the optimization process, Figures \ref{fig::overhang_cost_BCR_effect} and \ref{fig::overhang_overhang_BCR_effect} present the evolution of the cost function and the overhang constraint for several $\mu$. From these results, two elements are highlighted. First, a lower value of $\mu$ leads to a faster correction of the overhang constraint when broken due to the change in optimization hyperparameters. However, for lower $\mu$, more oscillations of the constraint value are observed. In some cases, the relaxation factor increases too fast, leading to new constraint breakage (see the curves for $\mu=0.6$ and $\mu=0.8$). This is especially evident in the fourth loop (iteration 151 to 200). Despite the observed differences in the first few loops,  the $\mu$ parameter has little impact on the optimization process's final cost function and constraint values. This can be explained as large changes in the material distribution happening in the first few loops, causing more constraint breakage; however, as the topology of the design settles down, there is almost no constraint breakage, implying that the effect of $\mu$ becomes negligible. Interestingly, the variation of the final cost function value with respect to $\mu$ is not monotonous. There was no clear explanation for this phenomenon other than the small perturbation caused by the difference of $\mu$ during the optimization process, which caused the optimizer to converge to different local minima.
Based on these observations, this parameter should be configured to prevent oscillation caused by a rapid increase in relaxation after all constraints are satisfied while still ensuring a quick response when constraints are violated. In these tests, values between 0.9 and 0.95 yielded optimal results.

\begin{figure}[!htpb]
	\centering
	\includegraphics[width=1\textwidth]{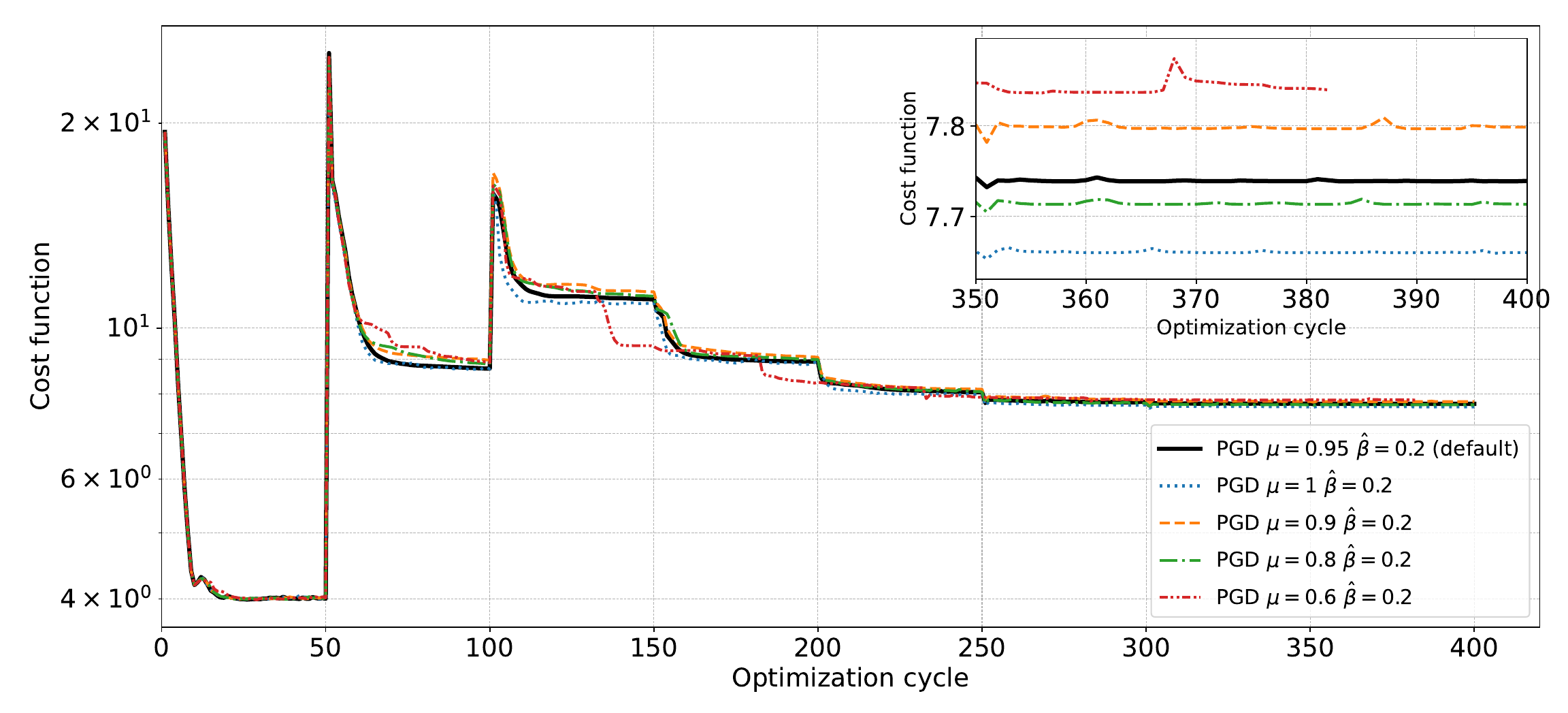}
	\caption{Comparison of the cost function evolution for different values of the broken constraint relaxation factor $\mu$.}
\label{fig::overhang_cost_BCR_effect}
\end{figure}

\begin{figure}[!htpb]
	\centering
	\includegraphics[width=1\textwidth]{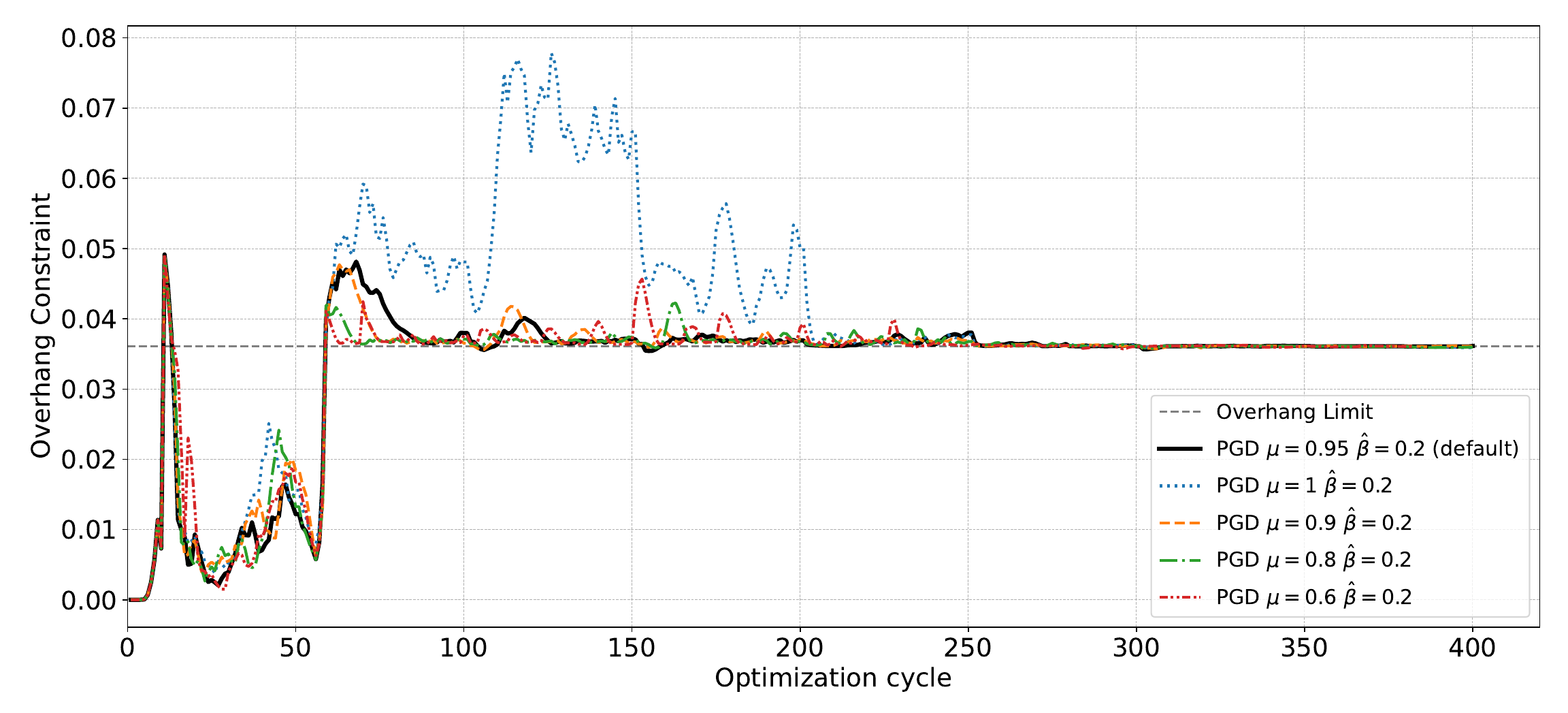}
	\caption{Comparison of the overhang constraint evolution for different values of the broken constraint relaxation factor $\mu$.}
\label{fig::overhang_overhang_BCR_effect}
\end{figure}

\subsubsection{Effect of the inertia parameter, $\hat{\beta}$}
The inertia term adds momentum to the optimization process and prevents the optimizer from getting stuck in local minima. Inertia can further improve the convergence speed in some cases \cite{nishioka2023inertial}, although it can also slow convergence for (mostly) convex problems. The effect of this parameter on the optimization process is presented in Figures \ref{fig::overhang_cost_inertia_effect} and \ref{fig::overhang_overhang_inertia_effect}.  These results show that increasing $\hat{\beta}$ decreases the final cost function. However, the increase can also lead to oscillations in the constraint values and larger constraint violations. In the first loop (cycles 1 to 50), for the two larger $\hat{\beta}$ values (0.3 and 0.4), the optimizer has trouble maintaining the overhang constraint. This, in turn, slows the convergence of the optimizer, as observed for $\hat{\beta}=0.4$ in Figure \ref{fig::overhang_cost_inertia_effect}. 
Conversely, the optimizer without inertia ($\hat{\beta}=0$) converges quickly at the beginning and seems more effective than optimizers without inertia. However, from the fourth loop onward, the cost function for this case starts exceeding that of cases with $\hat{\beta}>0$, and the differences become progressively larger, suggesting that the optimizer is stuck in a sub-optimal local minimum. The ideal value for this parameter depends on the cost function landscape and its smoothness. It should be chosen to balance the deterioration of the convergence when the problem is mostly convex and the prevention of converging to local minima as the hyperparameters are increased.

Finally, one can note that although the omission of the inertial factor ($\hat{\beta}=0$) can have a noticeable effect on the final cost function (5\,\% higher cost than with $\hat{\beta}=0.2$), the latter varied by at most 7\,\% when $\hat{\beta}>0$. Regarding results for the other two optimizer parameters, $\mu$ and $\epsilon_i$, the largest cost variation of approximately 10\,\% was observed when using the overly low $\epsilon_i$ value. Hence, we may conclude that the proposed PGD algorithm is only slightly sensitive to the optimizer parameter values. 
\begin{figure}[!htpb]
	\centering
	\includegraphics[width=1\textwidth]{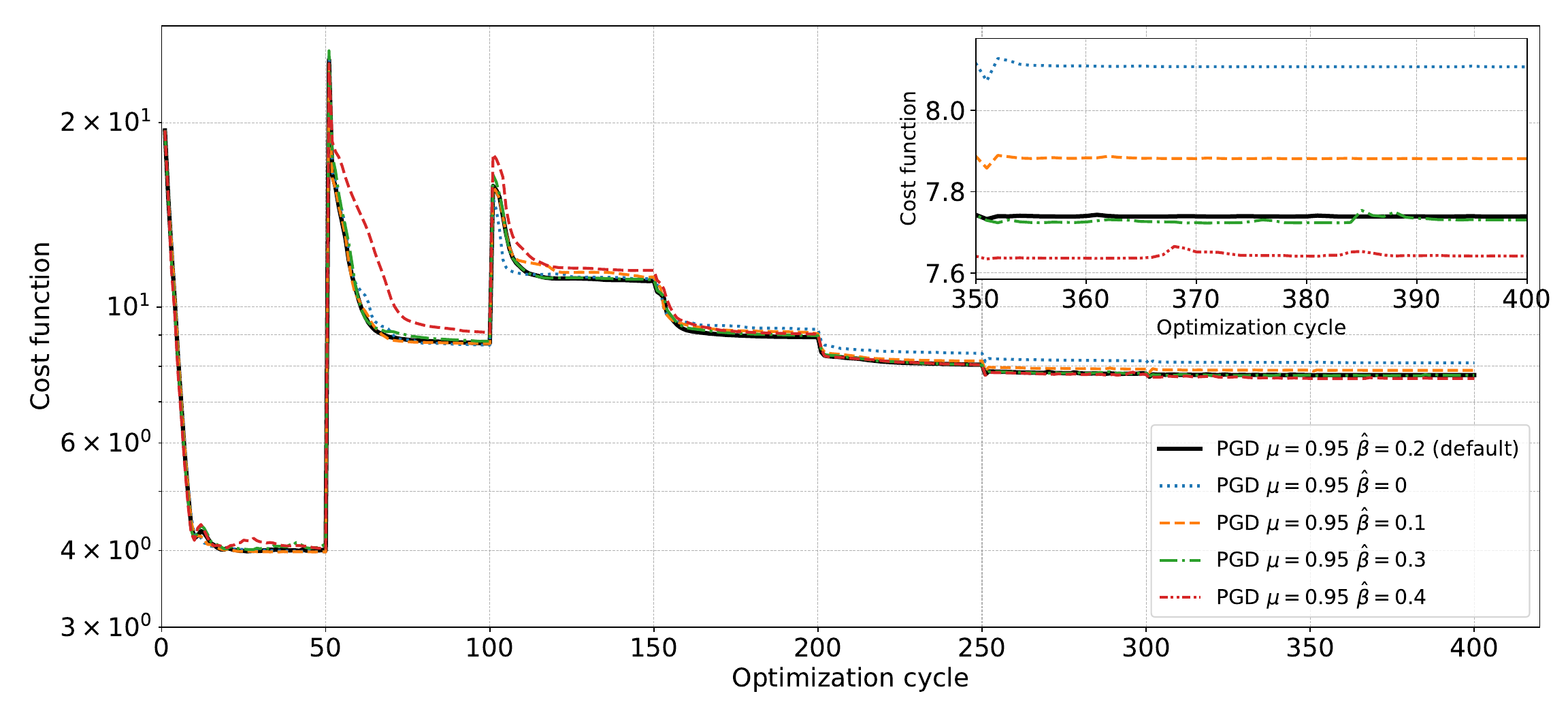}
	\caption{Comparison of the cost function evolution for different values of the inertia factor $\hat{\beta}$.}
\label{fig::overhang_cost_inertia_effect}
\end{figure}

\begin{figure}[!htpb]
	\centering
	\includegraphics[width=1\textwidth]{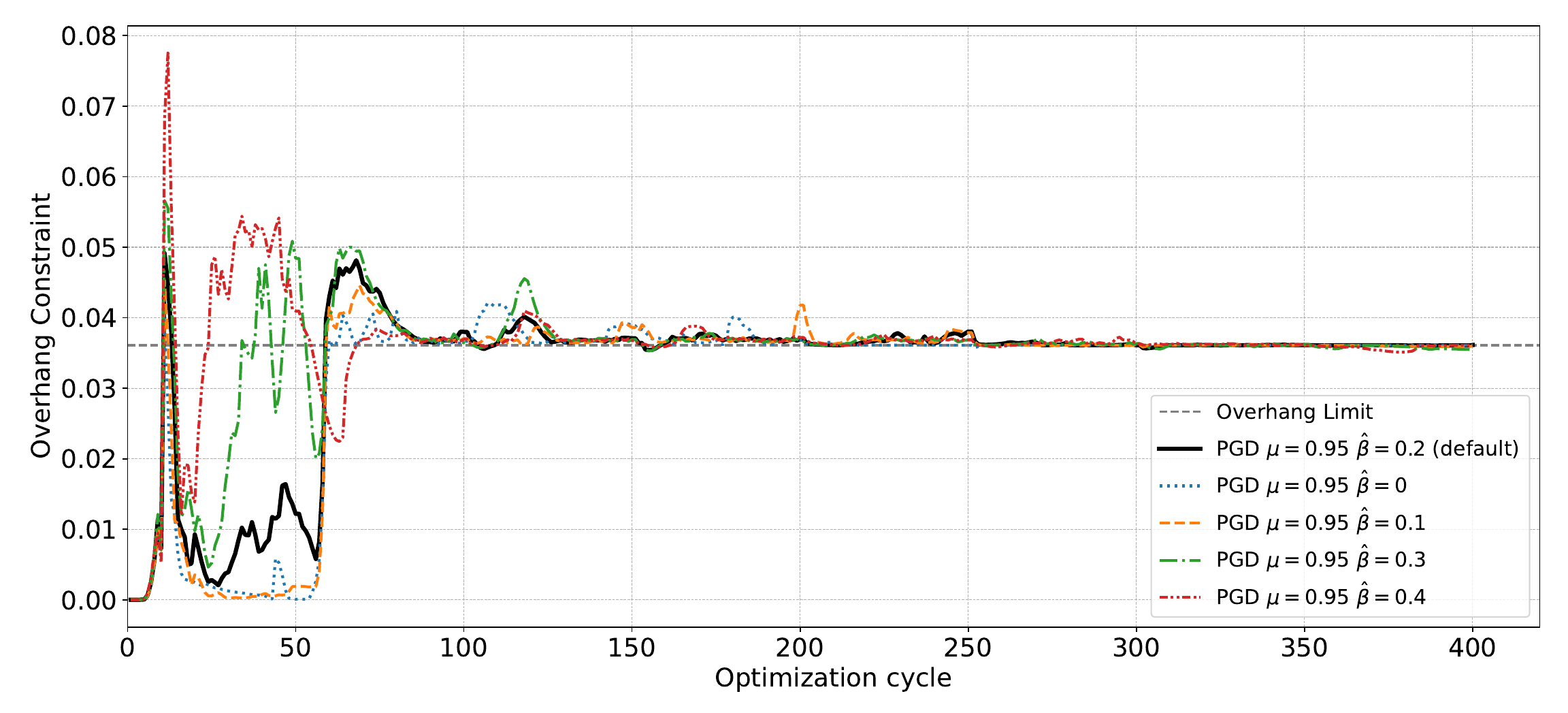}
	\caption{Comparison of the overhang constraint evolution for different values of the inertia factor $\hat{\beta}$.%
 }
\label{fig::overhang_overhang_inertia_effect}
\end{figure}
\FloatBarrier

\bmsection{Conclusion}
\label{sec::conclusion}
    

This study introduces two main modifications to the PGD algorithm: the incorporation of a bulk constraint manipulation combined with Schur complement to handle univariate constraints in the projection step and the decomposition of the update step into two distinct components relative to the constraint vector space. These innovations enable the algorithm to handle nonlinear constraints more effectively.
\begin{itemize} 
    \item \textbf{Active set with bulk constraint manipulation and efficient projection evaluation:} The modified PGD algorithm uses a modified active set algorithm that supports the bulk manipulation of constraints and ensures convergence with a fallback mechanism to single constraint manipulation. This algorithm also leverages the Schur complement to consider directly univariate constraints and accelerate the resolution of the projection step. This approach addresses issues associated with direct min--max clipping, ensuring an optimal projection. 
    \item \textbf{Splitting the update into two components:} the algorithm's robustness in the presence of nonlinear constraints is improved by decomposing the projected gradient into two components: one orthogonal to all constraints variations and the other within the vector space defined by the span of the constraints gradient. This decomposition allows for a dynamic scaling of the orthogonal component and considers the variation of the Lagrangian gradient and the state (broken or not) of the nonlinear constraints. The step size adjustment ensures an effective handling of the constraints, leading to more stable optimization and mitigating constraint violation. 
\end{itemize}

By integrating these improvements, the proposed PGD algorithm achieves a more robust and efficient optimization process capable of handling nonlinear constraints with increased accuracy and stability. The algorithm is applied to a topology optimization problem involving the design of a heat sink. The results presented in Section \ref{sec::application} demonstrate that the proposed algorithm performs comparably or, in certain cases, slightly better than the well-known and established MMA algorithm \cite{svanberg1987method}.
The MMA method relies on precisely tuning the asymptote's moving limit parameter to ensure fast and reliable convergence. In contrast, the proposed PGD algorithm requires minimal tuning to achieve effective results. Specifically, setting the inertia parameter $\hat{\beta}$ between 0 and 0.3 assists in escaping some local minima without significantly slowing convergence, improving the final cost function by up to 5\,\%. However, increasing this parameter too much causes the convergence to slow down. Also, two additional parameters are available for constraints management: a relaxation factor for broken constraints ($\mu$) and a tolerance for each constraint ($\epsilon_i$). These parameters moderately influence the optimization process; using a lower tolerance ($\epsilon_i$) will enforce a stricter compliance with the constraint limit, at the cost of a slower convergence and a potentially suboptimal final solution if insufficient iterations are allowed. As for $\mu$, its value has a minimal impact on the final cost function (at most 2\,\% variation for the considered case); however, setting $\mu$ too high can lead to large constraint violations, whereas a too low value can cause oscillation in nonlinear constraints. Still, we show that the optimizer converges across various parameter settings without significantly affecting the final cost function and constraint values (the largest final cost variation (10\,\%) was caused by the use of an overly small tolerance). Consequently, the proposed PGD algorithm does not require fine-tuning the parameters, although adjustments can occasionally enhance the final solution obtained.

While the developed PGD algorithm performed effectively in most cases, specific scenarios highlighted potential areas for improvement. For instance, applying the overhang constraint directly to the solid fraction indicator ($\Bar{\phi}$) instead of the filtered design variable ($\hat{\phi}$) made the optimization process much more difficult since the overhang constraint passes through the Heaviside projection, amplifying nonlinearities. Additionally, the increase in the sharpness parameter $\lambda$ between inner loops causes breakage of the overhang constraint without changing the design variable. As $\lambda$ increases, correcting the constraint following the change in $\lambda$ becomes more severe and requires more iterations. This highlights that limitations still exist for very nonlinear cases. A possible improvement could be to implement a relaxation for broken constraints that dynamically adjusts based on the error in the constraints instead of solely relying on the consecutive number of iterations where constraints are broken. This modification could ensure that the jump in the overhang constraint due to a change in $\lambda$ between inner loops is directly captured and corrected more rapidly.

The proposed PGD algorithm’s applicability goes beyond topology optimization. This type of first-order algorithm will excel in applications involving large design variable spaces, such as topology optimization. Among other potential applications, using the PGD algorithm to craft adversarial prompts in large language models has shown great potential \cite{geisler2024attacking}. These adversarial prompts usually use the embedding space of the token and the token space 
to craft adversarial prompts. This type of problem can have large dimensionality (easily reaching millions of decision variables), making the current PGD algorithm a promising avenue.

\bmsection*{Author contributions}

\textbf{Lucka Barbeau:} Conceptualization (lead); Data Curation (lead); Formal Analysis (lead); Methodology (lead); Software (lead); Validation (lead); Visualization (lead); Writing – Original Draft Preparation (lead); Writing – Review \& Editing (equal). 
\textbf{Marc-Étienne Lamarche-Gagnon:} Funding Acquisition (lead); Software (supporting); Formal Analysis (supporting); Visualization (supporting); Writing – Original Draft Preparation (equal); Project Administration (lead); Writing – Review \& Editing (equal). 
\textbf{Florin Ilinca:} Supervision (lead); Writing – Review \& Editing (supporting).


\bmsection*{Financial disclosure}
This work was partially supported by the National Research Council of Canada METALTec industrial research group, as well as by the Office of Energy Research and Development of Canada (grant NRC-23-128) and by the Centre Québecois de Recherche et de Développement de l’Aluminium (grant CNRC1185).

\bmsection*{Conflict of interest}

The authors declare no potential conflict of interests.

\bibliography{wileyNJD-AMA}



\appendix
\section{Active set frequency of single constraints manipulation}
\label{appendix::frequency_6_6c_fails}
This appendix presents results for the frequency of single constraint manipulation in the algorithm presented in Section \ref{sec::active_set}.
To test this, the algorithm solves a simple convex optimization problem in the form of:
\begin{align}
    \mathcal{C}(\phi)=\sum_i^k|B_i|(\phi_i-B_i)^4.
\end{align}
Here $B$ is vector of length $k$ and it element $B_i$ are uniformly distributed random number between -10 and 10.
This optimization problem has $m$ linear inequality constraints in the form of :
\begin{align}
    \mathcal{A}\bm{\phi}\leq \bm{a},
\end{align}
where $\mathcal{A}$ is an $m$ by $k$ matrix of random number between -1 and 1, while $\bm{a}$ is the vector of $m$ positive inequality constant randomly selected between 0 and 1. The problem also has bound constraints in the form of:
\begin{align}
   -10\leq \bm{\phi}\leq 10.
\end{align}.

Using this general form, multiple optimization problems are tested, and the frequency at which the active set algorithm requires the fallback mechanism (when step 6 or step 6c condition is not satisfied) is evaluated. To do so, 30000 different cases are generated and optimized for each given pair of a number of degrees of freedom ($k$) and a number of global inequality constraints ($m$).

\begin{table}[h!]
    \centering
    \begin{tabular}{|c|c|c|c|c|}
        \hline
        & \multicolumn{4}{c|}{\textbf{Number of Degrees of Freedom}} \\
        \hline
        \textbf{Number of constraints} & \textbf{5} & \textbf{10} & \textbf{20} & \textbf{40} \\
        \hline
        \textbf{5} & 491 732 & 746 735 & 997 143 & 1 230 837  \\
        \textbf{10} & - & 590 326 & 882 156 &1 147 100
  \\
        \textbf{20}&- & -& 689 387& 990 823  \\
        \textbf{40} & -& -& -& 764 242 \\
        \hline
    \end{tabular}
    \caption{Total number of PGD iterations for different numbers of degrees of freedom (DOF) and number of constraints.}
    \label{tab:dof_vs_constraints_1}
\end{table}

\begin{table}[h!]
    \centering
    \begin{tabular}{|c|c|c|c|c|}
        \hline
        & \multicolumn{4}{c|}{\textbf{Number of Degrees of Freedom}} \\
        \hline
        \textbf{Number of constraints} & \textbf{5} & \textbf{10} & \textbf{20} & \textbf{40} \\
        \hline
        \textbf{5} & 1 198 364 & 1 852 362 &2 563 174 & 3 306 136 \\
        \textbf{10} & - & 2 085 808 & 2 881 787 &3 675 257  \\
        \textbf{20}&- & -& 3 304 679& 4 067 164 \\
        \textbf{40}& -& -& -& 4 760 507 \\
        \hline
    \end{tabular}
    \caption{Total number of projections (step 5 or step 6b) for different numbers of degrees of freedom (DOF) and number of constraints.}
    \label{tab:dof_vs_constraints_2}
\end{table}

\begin{table}[h!]
    \centering
    \begin{tabular}{|c|c|c|c|c|}
        \hline
        & \multicolumn{4}{c|}{\textbf{Number of Degrees of Freedom}} \\
        \hline
        \textbf{Number of constraints} & \textbf{5} & \textbf{10} & \textbf{20} & \textbf{40}  \\
        \hline
        \textbf{5} & 292& 6 & 0 & 0  \\
        \textbf{10} & - & 495 & 0 & 0\\
        \textbf{20}&- & -&160 & 0 \\
        \textbf{40} &- & -& -&2 \\
        \hline
    \end{tabular}
    \caption{Total number of times the condition of step 6 was not satisfied for different numbers of degrees of freedom and number of constraints.}
    \label{tab:dof_vs_constraints_3}
\end{table}

\begin{table}[h!]
    \centering
    \begin{tabular}{|c|c|c|c|c|}
        \hline
        & \multicolumn{4}{c|}{\textbf{Number of Degrees of Freedom}} \\
        \hline
        \textbf{Number of constraints} & \textbf{5} & \textbf{10} & \textbf{20} & \textbf{40}  \\
        \hline
        \textbf{5} & 27 & 0 & 0&  0\\
        \textbf{10} & - & 7 & 0 &0 \\
        \textbf{20}&- & -&0 &0  \\
        \textbf{40} &-& -&- & 0 \\
        \hline
    \end{tabular}
    \caption{Total number of times the condition of step 6c was not satisfied for different numbers of degrees of freedom (DOF) and number of constraints.}
    \label{tab:dof_vs_constraints_4}
\end{table}

From the results presented in Tables \ref{tab:dof_vs_constraints_1} to \ref{tab:dof_vs_constraints_4}, it is possible to see that it is extremely rare to observe that condition 6 is not satisfied and even more rare to observe that condition 6c is not satisfied. It is also possible to observe that it mostly happens in very constrained cases (cases with the same number of global inequality constraints as degrees of freedom). Furthermore, these results show that for these cases, the number of active set iterations per gradient projection required varies between 2 and 5 on average, increasing with the number of constraints.



\end{document}